\documentclass[letterpaper,12pt]{amsart}
\usepackage{amssymb}
\usepackage{amsfonts}
\usepackage{amsmath}

\newtheorem{theorem}{Theorem}
\theoremstyle{plain}

\newtheorem{corollary}{Corollary}

\newtheorem{definition}{Definition}

\newtheorem{lemma}{Lemma}

\newtheorem{remark}{Remark}

\numberwithin{equation}{section}
\input{tcilatex}

\begin{document}
\title{Enriched pro-categories and shapes}
\author{Nikica Ugle\v{s}i\'{c}}
\curraddr{23287 Veli R\aa t, Dugi Otok, Croatia}
\email{uglesic@pmfst.hr}
\date{December 4, 2017}

\begin{abstract}
Given a category $\mathcal{C}$ and a directed partially ordered set $J$, a
certain category $pro^{J}$-$\mathcal{C}$ on inverse systems in $\mathcal{C}$
is constructed such that the ordinary pro-category $pro$-$\mathcal{C}$ is
the most special case of a singleton $J$ $\equiv $\{1\}$.$ Further, the
known pro$^{\ast }$-category $pro^{\ast }$-$\mathcal{C}$ becomes $pro^{%
\mathbb{N}}$-$\mathcal{C}$. Moreover, given a pro-reflective category pair $(%
\mathcal{C},\mathcal{D})$, the $J$-shape category $Sh_{(\mathcal{C},\mathcal{%
D})}^{J}$ and the corresponding $J$-shape functor $S^{J}$ are constructed
which, in mentioned special cases, become the well known ones. Among several
important properties, the continuity theorem for a $J$-shape category is
established. It implies the \textquotedblleft $J$-shape
theory\textquotedblright\ is a genuine one such that the shape and the
coarse shape theory are its very special examples.
\end{abstract}

\maketitle

\section{Introduction}

The shape theory, from the very begining, has been an operable extension and
generalization of the homotopy theory to the class of all (locally bad)
topological spaces. Since Borsuk's paper [1] and book [2], many articles
([6], [7], [16], [20], [22], [26], [27] are some of the most fundamental)
and several books ([3], [8], [10], [24]) concerning shape theory were
written almost in the first decade already. By attempting to describe the
shape theory (standard and abstract) as an axiomatic homotopy theory
(founded by D. G. Quillen, [28]), the strong shape theory has been obtained
([11], [5], [12]). At the same time some shape theorists introduced and
considered several classifications of metrizable compacta coarser than the
shape type. The most interesting of them are the Borsuk's quasi-equivalence
[4] and Marde\v{s}i\'{c} $S$-equivalence [20]. They were further studied by
the author and some others ([9], [13], [15], [17], [35] and, as a survey,
[30]). On that line, the most important has become a certain uniformization
of the $S$-equivalence, called the $S^{\ast }$-equivalence, which admits a
categorical characterization, [25]. Moreover, it admits (genuine and
different; [31], [33]) generalizations to all topological spaces as well as
to any abstract categorical framework ([19], [34], [36]]), and all the well
known shape invariants remain as the invariants of the both generalizations
(in addition, [18] and [29]).

In this paper we generalize the generalization introduced in [19], the
coarse shape theory, so that it and the shape theory as well become the very
special cases of the new, so called, $J$-shape theory.

A part of the idea came from the recently founded quotient shape theory for
a concrete category, [32]. Namely, figuratively speaking, the quotient
shapes of an object are \textquotedblleft its (changeable)
pictures\textquotedblright\ depending on the distance of the
\textquotedblleft view point\textquotedblright\ which is determined by a
\textquotedblleft reciprocal\textquotedblright\ infinite cardinality (larger
cardinal - closer distance, i.e., finer picture, and comparing them to the
objects of lower cardinalities). This role hereby overtakes a directed
partially ordered set $J$ (larger set $J$ - larger distance, i.e., coarser
picture, and the comparing objects are those of $\mathcal{D}$). In order to
realize this idea, we have followed the construction of the coarse shape
category obtained in [19]. Given a category $\mathcal{C}$ and a directed
partially ordered set $J$, in the first step (Section 3), each morphism set $%
(inv$-$\mathcal{C})(\boldsymbol{X},\boldsymbol{Y})$ is essentially enriched,
according to $J$, to the set $(inv^{J}$-$\mathcal{C})(\boldsymbol{X},%
\boldsymbol{Y})$ making a new category $inv^{J}$-$\mathcal{C}$ (with the
same object class - all inverse systems in $\mathcal{C}$). In the second
step, on each set $(inv^{J}$-$\mathcal{C})(\boldsymbol{X},\boldsymbol{Y})$
an equivalence relation is defined, according to $J$, that is compatible
with the composition so that there is the corresponding quotient category $%
(inv^{J}$-$\mathcal{C})/\sim $, denoted by $pro^{J}$-$\mathcal{C}$. In the
trivial case $J=\{1\}$, $pro^{\{1\}}$-$\mathcal{C}=pro$-$\mathcal{C}$, while
in the case of $J=\mathbb{N}$, $pro^{\mathbb{N}}$-$\mathcal{C}=pro^{\ast }$-$%
\mathcal{C}$ (of [19]). Then, for a suitable pair $\boldsymbol{X},%
\boldsymbol{Y}$ and an enough large $J$, in the set $(pro^{J}$-$\mathcal{C})(%
\boldsymbol{X},\boldsymbol{Y})$ may exist an isomorphism, while there is no
isomorphism in the set ($pro$-$\mathcal{C})(\boldsymbol{X},\boldsymbol{Y})$.
Finally, in the third step (Section 4), given a pro-reflective subcategory
pair $\mathcal{D}\subseteq \mathcal{C})$, the construction of the
appropriate $J$-shape category $Sh_{(\mathcal{C},\mathcal{D})}^{J}$ and the $%
J$-shape functor $S^{J}:\mathcal{C}\rightarrow Sh_{(\mathcal{C},\mathcal{D}%
)}^{J}$ follows by the usual standard pattern. Clearly, in the mentioned
special case, $Sh_{(\mathcal{C},\mathcal{D})}^{\{1\}}=Sh_{(\mathcal{C},%
\mathcal{D})}$ (the abstract shape category of [24]) and $Sh_{(\mathcal{C},%
\mathcal{D})}^{\mathbb{N}}=Sh_{(\mathcal{C},\mathcal{D})}^{?}$ (the abstract
coarse shape category of [19]) having their realizing categories $%
pro^{\{1\}} $-$\mathcal{D}=pro$-$\mathcal{D}$ and $pro^{\mathbb{N}}$-$%
\mathcal{D}=pro^{\ast }$-$\mathcal{D}$.

In Section 5 we have proven the continuity theorem for every $J$-shape
category. It strongly confirms that the $J$-shape theory is a genuine shape
theory. At the end (Section 6) we have proven the full analogue of the well
known Morita lemma of [26] that characterizes an isomorphism of $pro^{J}$-$%
\mathcal{C}$, which is then very useful for characterizing a $J$-shape
isomorphism in the corresponding realizing category $pro^{J}$-$\mathcal{D}$.

Of course, the whole of this should be firstly applied to the pro-reflective
category pair $(HTop;HPol)$ and to its subpair $(HcM,HcPol)$ (where only
sequential expansions are needed).

\section{Preliminaries}

We assume that the notion of a pro-category is well known as well as the
basics of the (abstract) shape theory, especially, via the inverse systems
approach due to Marde\v{s}i\'{c} and Segal, [24]. For the sake of
completeness, we shall briefly recall the needed notions and main facts
concerning a pro$^{\ast }$-category and the coarse shape obtained in [19].
The category language follows [14].

Let $\mathcal{C}$ be a category, and let $inv$-$\mathcal{C}$ be the
corresponding inv-category. Given a pair $\boldsymbol{X}$, $\boldsymbol{Y}$
of inverse systems in $\mathcal{C}$, a $\ast $\emph{-morphism} (originally,
an $S^{\ast }$-morphism) of $\boldsymbol{X}$ to $\boldsymbol{Y}$, denoted by

$(f,f_{\mu }^{n}):\boldsymbol{X}=(X_{\lambda },p_{\lambda \lambda ^{\prime
}},\Lambda )\rightarrow (Y_{\mu },q_{\mu \mu ^{\prime }},M)=\boldsymbol{Y}$,

\noindent is an ordered pair consisting of a function $f:M\rightarrow
\Lambda $ (the \emph{index function}) and, for each $\mu \in M$, of a
sequence $(f_{\mu }^{n})$ of $\mathcal{C}$-morphisms $f_{\mu }^{n}:X_{f(\mu
)}\rightarrow Y_{\mu }$, $n\in \mathbb{N}$, satisfying the following
condition:

$(\forall \mu \leq \mu ^{\prime }$ in $M)(\exists \lambda \in \Lambda ,$ $%
\lambda \geqslant f(\mu ),f\left( \mu ^{\prime }\right) (\exists n\in 
\mathbb{N}(\forall n^{\prime }\geqslant n)$

$f_{\mu }^{n^{\prime }}p_{f(\mu )\lambda }=q_{\mu \mu ^{\prime }}f_{\mu
^{\prime }}^{n^{\prime }}p_{f(\mu ^{\prime })\lambda }$.

\noindent Clearly, the equality then holds for every $\lambda ^{\prime }\geq
\lambda $ as well. If the index function $f$ is increasing and, for every
pair $\mu \leq \mu ^{\prime }$, one may put $\lambda =f(\mu ^{\prime })$,
then $(f,f_{\mu }^{n})$ is said to be a \emph{simple} $\ast $-morphism. If,
in addition, $M=\Lambda $ and $f=1_{\Lambda }$, then $(1_{\Lambda
},f_{\lambda }^{n})$ is said to be a \emph{level} $\ast $-morphism. Finally,
a $\ast $-morphism $(f,f_{\mu }^{n}):\boldsymbol{X}\rightarrow \boldsymbol{Y}
$ is said to be \emph{commutative} whenever, for every pair $\mu \leq \mu
^{\prime }$, one may put $n=1$.

\noindent If $\boldsymbol{Y}=\boldsymbol{X}$, the \emph{identity} $\ast $%
-morphism $(1_{\Lambda },1_{\lambda }^{n}):\boldsymbol{X}\rightarrow 
\boldsymbol{X}$ is defined by putting, for each $\lambda \in \Lambda $ and
every $n\in \mathbb{N}$, $1_{\lambda }^{n}\equiv 1_{\lambda }$ to be the
identity $\mathcal{C}$-morphism on $X_{\lambda }$. The \emph{composition} of 
$\left( f,f_{\mu }^{n}\right) :\boldsymbol{X}\rightarrow \boldsymbol{Y}$
with a $\ast $-morphism $\left( g,g_{\nu }^{n}\right) :\boldsymbol{Y}%
\rightarrow \boldsymbol{Z}=(Z_{\nu },r_{\nu \nu ^{\prime }},N)$ is defined by

$(h=fg,h_{\nu }^{n}=g_{\nu }^{n}f_{g\left( \nu \right) }^{n}):\boldsymbol{X}%
\rightarrow \boldsymbol{Z}$.

\noindent The category $inv^{\ast }$-$\mathcal{C}$ is now defined by putting 
$Ob(inv^{\ast }$-$\mathcal{C})=Ob(inv$-$\mathcal{C})$ and $(inv^{\ast }$-$%
\mathcal{C})(\boldsymbol{X},\boldsymbol{Y})$ to be the set of all $\ast $%
-morphisms of $\boldsymbol{X}$ to $\boldsymbol{Y}$.

A $\ast $-morphism $(f,f_{\mu }^{n}):\boldsymbol{X}\rightarrow \boldsymbol{Y}
$ is said to be \emph{equivalent to} a $\ast $-morphism $(f^{\prime },f_{\mu
}^{\prime n}):\boldsymbol{X}\rightarrow \boldsymbol{Y}$, denoted by $%
(f,f_{\mu }^{n})\sim (f^{\prime },f_{\mu }^{\prime n})$, if

$(\forall \mu \in M)(\exists \lambda \in \Lambda $, $\lambda \geqslant f(\mu
),f^{\prime }(\mu ))((\exists n\in \mathbb{N})(\forall n^{\prime }\geqslant
n)$

$f_{\mu }^{n^{\prime }}p_{f(\mu )\lambda }=f_{\mu }^{\prime n^{\prime
}}p_{f^{\prime }(\mu )\lambda }$.

\noindent The equality holds for every $\lambda ^{\prime }\geq \lambda $ as
well. The relation $\sim $ is an equivalence relation on each set $%
(inv^{\ast }$-$\mathcal{C})(\boldsymbol{X},\boldsymbol{Y}),$ and the
equivalence class $[(f,f_{\mu }^{n})]$ of $(f,f_{\mu }^{n}):\boldsymbol{X}%
\rightarrow \boldsymbol{Y}$ is briefly denoted by $\boldsymbol{f}^{\ast }$.
The equivalence relation $\sim $ is compatible with the composition, i.e.,
if $(f,f_{\mu }^{n})\sim (f^{\prime },f_{\mu }^{\prime n})\boldsymbol{\ }$%
and $(g,g_{\nu }^{n})\sim (g^{\prime },g_{\nu }^{\prime n}):\boldsymbol{Y}%
\rightarrow \boldsymbol{Z}$, then

$(g,g_{\nu }^{n})(f,f_{\mu }^{n})\sim (g^{\prime },g_{\nu }^{\prime
n})(f^{\prime },f_{\mu }^{\prime n}):\boldsymbol{X}\rightarrow \boldsymbol{Z}
$.$.$

The pro$^{\ast }$-category $pro^{\ast }$-$\mathcal{C}$ is now defined to be
the quotient category $(inv^{\ast }$-$\mathcal{C})/\sim $, i.e.,

$Ob(pro^{\ast }$-$\mathcal{C})=Ob(inv^{\ast }$-$\mathcal{C})$ \quad ($%
=Ob(inv $-$\mathcal{C})=Ob(pro$-$\mathcal{C})$),

$(pro^{\ast }$-$\mathcal{C})(\boldsymbol{X},\boldsymbol{Y})=(inv^{\ast }$-$%
\mathcal{C})(\boldsymbol{X},\boldsymbol{Y})/\sim $ =

$=\{\boldsymbol{f}^{\ast }=[(f,f_{\mu }^{n})]\mid (f,f_{\mu }^{n}):%
\boldsymbol{X}\rightarrow \boldsymbol{Y}\}$.

Finally, there exists a faithful functor $\underline{I}:pro$-$\mathcal{C}%
\rightarrow pro^{\ast }$-$\mathcal{C}$, keeping the object fixed, such that,
for every $\boldsymbol{f}=[(f,f_{\mu })]\in (pro$-$\mathcal{C})(\boldsymbol{%
XY})$,

$\underline{I}\left( \boldsymbol{f}\right) \equiv \boldsymbol{f}^{\ast
}=[(f,f_{\mu }^{n})]\in (pro^{\ast }$-$\mathcal{C})(\boldsymbol{X},%
\boldsymbol{Y}),$

\noindent where, for each $\mu \in M$ and every $n\in \mathbb{N}$, $f_{\mu
}^{n}=f_{\mu }$.

Let $\mathcal{D}$ be a full (not essential, but a convenient condition) and
pro-reflective subcategory of $\mathcal{C}$. Let $\boldsymbol{p}%
:X\rightarrow \boldsymbol{X}$ and $\boldsymbol{p}^{\prime }:X\rightarrow 
\boldsymbol{X}^{\prime }$ be $\mathcal{D}$-expansions of the same object $X$
of $\mathcal{C}$, and let $\boldsymbol{q}:Y\rightarrow \boldsymbol{Y}$ and $%
\boldsymbol{q}^{\prime }:Y\rightarrow \boldsymbol{Y}^{\prime }$ be $\mathcal{%
D}$-expansions of the same object $Y$ of $\mathcal{C}$. Then there exist two
canonical (unique) isomorphisms $\boldsymbol{i}:\boldsymbol{X}\rightarrow 
\boldsymbol{X}^{\prime }$ and $\boldsymbol{j}:\boldsymbol{Y}\rightarrow 
\boldsymbol{Y}^{\prime }$ of $pro$-$\mathcal{D}$. Consequently, $\boldsymbol{%
i}^{\ast }\equiv \underline{I}(\boldsymbol{i}):\boldsymbol{X}\rightarrow 
\boldsymbol{X}^{\prime }$ and $\boldsymbol{j}^{\ast }\equiv \underline{I}(%
\boldsymbol{j}):\boldsymbol{Y}\rightarrow \boldsymbol{Y}^{\prime }$ are
isomorphisms of $pro^{\ast }$-$\mathcal{D}$. A morphism $\boldsymbol{f}%
^{\ast }:\boldsymbol{X}\rightarrow \boldsymbol{Y}$ is said to be $pro^{\ast
} $-$\mathcal{D}$ \emph{equivalent to} a morphism $\boldsymbol{f}^{\prime
\ast }:\boldsymbol{X}^{\prime }\rightarrow \boldsymbol{Y}^{\prime }$,
denoted by $\boldsymbol{f}^{\ast }\sim \boldsymbol{f}^{\prime \ast }$, if
the following diagram in $pro^{\ast }$-$\mathcal{D}$ commutes:

$%
\begin{array}{ccc}
\boldsymbol{X} & \overset{\boldsymbol{i}^{\ast }}{\longrightarrow } & 
\boldsymbol{X}^{\prime } \\ 
\boldsymbol{f}^{\ast }\downarrow &  & \downarrow \boldsymbol{f}^{\prime \ast
} \\ 
\boldsymbol{Y} & \overset{\boldsymbol{j}^{\ast }}{\longrightarrow } & 
\boldsymbol{Y}^{\prime }%
\end{array}%
$.

\noindent According to the analogous facts in $pro$-$\mathcal{D}$, and since 
$\underline{I}$ is a functor, it defines an equivalence relation on the
appropriate subclass of $Mor(pro^{\ast }$-$\mathcal{D})$, such that $%
\boldsymbol{f}^{\ast }\sim \boldsymbol{f}^{\prime \ast }$ and $\boldsymbol{g}%
^{\ast }\sim \boldsymbol{g}^{\prime \ast }$ imply $\boldsymbol{g}^{\ast }%
\boldsymbol{f}^{\ast }\sim \boldsymbol{g}^{\prime \ast }\boldsymbol{f}%
^{\prime \ast }$ whenever it is defined. The equivalence class of an $%
\boldsymbol{f}^{\ast }$ is denoted by $\left\langle \boldsymbol{f}^{\ast
}\right\rangle $. Further, given $\boldsymbol{p}$, $\boldsymbol{p}^{\prime }$%
, $\boldsymbol{q}$, $\boldsymbol{q}^{\prime }$ and $\boldsymbol{f}^{\ast }$
as above, there exists a unique $\boldsymbol{f}^{\prime \ast }$ ($=%
\boldsymbol{j}^{\ast }\boldsymbol{f}^{\ast }(\boldsymbol{i}^{\ast })^{-1}$)
such that $\boldsymbol{f}^{\ast }\sim \boldsymbol{f}^{\prime \ast }$. Then
the (\emph{abstract}) \emph{coarse} \emph{shape category} $Sh_{(\mathcal{C},%
\mathcal{D})}^{\ast }$ \emph{for} $(\mathcal{C},\mathcal{D})$ is defined as
follows. The objects of $Sh_{(\mathcal{C},\mathcal{D})}^{\ast }$ are all the
objects of $\mathcal{C}$. A morphism $F^{\ast }\in Sh_{(\mathcal{C},\mathcal{%
D})}^{\ast }(X,Y)$ is the $(pro^{\ast }$-$\mathcal{D})$-equivalence class $%
\left\langle \boldsymbol{f}^{\ast }\right\rangle $ of a morphism $%
\boldsymbol{f}^{\ast }:\boldsymbol{X}\rightarrow \boldsymbol{Y}$, with
respect to any choice of a pair of $\mathcal{D}$-expansions $\boldsymbol{p}%
:X\rightarrow \boldsymbol{X}$, $\boldsymbol{q}:Y\rightarrow \boldsymbol{Y}$.
In other words, a \emph{coarse shape morphism} $F^{\ast }:X\rightarrow Y$ is
given by a diagram

$%
\begin{array}{ccc}
\boldsymbol{X} & \overset{\boldsymbol{p}}{\longleftarrow } & X \\ 
\boldsymbol{f}^{\ast }\downarrow & F^{\ast } &  \\ 
\boldsymbol{Y} & \overset{\boldsymbol{q}}{\longleftarrow } & Y%
\end{array}%
$.

\noindent The \emph{composition} of an $F^{\ast }:X\rightarrow Y$, $F^{\ast
}=\left\langle \boldsymbol{f}^{\ast }\right\rangle $ and a $G^{\ast
}:Y\rightarrow Z$, $G^{\ast }=\left\langle \boldsymbol{g}^{\ast
}\right\rangle $, is defined by any pair of their representatives, i.e. $%
G^{\ast }F^{\ast }:X\rightarrow Z$, $G^{\ast }F^{\ast }=\left\langle 
\boldsymbol{g}^{\ast }\boldsymbol{f}^{\ast }\right\rangle $. The\emph{\
identity coarse shape morphism} on an object $X$, $1_{X}^{\ast
}:X\rightarrow X$, is the $(pro^{\ast }$-$\mathcal{D})$-equivalence class $%
\left\langle \boldsymbol{1}_{\boldsymbol{X}}^{\ast }\right\rangle $ of the
identity morphism $\boldsymbol{1}_{\boldsymbol{X}}^{\ast }$ of $pro^{\ast }$-%
$\mathcal{D}$.

For every $\mathcal{C}$-morphism $f:X\rightarrow Y$ and every pair of $%
\mathcal{D}$-expansions $\boldsymbol{p}:X\rightarrow \boldsymbol{X}$, $%
\boldsymbol{q}:Y\rightarrow \boldsymbol{Y}$, there exists an $\boldsymbol{f}%
^{\ast }:\boldsymbol{X}\rightarrow \boldsymbol{Y}$ of $pro^{\ast }$-$%
\mathcal{D}$, such that the following diagram in $pro^{\ast }$-$\mathcal{C}$
commutes:

$%
\begin{array}{ccc}
\boldsymbol{X} & \overset{\boldsymbol{p}}{\longleftarrow } & X \\ 
\boldsymbol{f}^{\ast }\downarrow &  & \downarrow f \\ 
\boldsymbol{Y} & \overset{\boldsymbol{q}}{\longleftarrow } & Y%
\end{array}%
$.

\noindent (Hereby, $\mathcal{C}\subseteq pro$-$\mathcal{C}$ are considered
to be the subcategories of $pro^{\ast }$-$\mathcal{C}$!) The same $f$ and
another pair of $\mathcal{D}$-expansions $\boldsymbol{p}^{\prime
}:X\rightarrow \boldsymbol{X}^{\prime }$, $\boldsymbol{q}^{\prime
}:Y\rightarrow \boldsymbol{Y}^{\prime }$ yield an $\boldsymbol{f}^{\prime
\ast }:\boldsymbol{X}^{\prime }\rightarrow \boldsymbol{Y}^{\prime }$ in $%
pro^{\ast }$-$\mathcal{D}$. Then, however, $\boldsymbol{f}^{\ast }\sim 
\boldsymbol{f}^{\prime \ast }$ in $pro^{\ast }$-$\mathcal{D}$ must hold.
Thus, every morphism $f\in \mathcal{C}(X,Y)$ yields a $(pro^{\ast }$-$%
\mathcal{D})$-equivalence class $\left\langle \boldsymbol{f}^{\ast
}\right\rangle $, i.e., a coarse shape morphism $F^{\ast }\in Sh_{(\mathcal{C%
},\mathcal{D})}^{\ast }(X,Y)$. Therefore, by putting $S^{\ast }(X)=X$, $X\in
Ob\mathcal{C}$, and $S^{\ast }(f)=F^{\ast }=\left\langle \boldsymbol{f}%
^{\ast }\right\rangle $, $f\in Mor\mathcal{C}$, a unique functor

$S_{(\mathcal{C},\mathcal{D})}^{\ast }:\mathcal{C}\rightarrow Sh_{(\mathcal{C%
},\mathcal{D})}^{\ast }$,

\noindent called the \emph{abstract} \emph{coarse} \emph{shape functor}, is
defined. Moreover, the functor $S_{(\mathcal{C},\mathcal{D})}^{\ast }$
factorizes as $S_{(\mathcal{C},\mathcal{D})}^{\ast }=I_{(\mathcal{C},%
\mathcal{D})}S_{(\mathcal{C},\mathcal{D})}$, where $S_{(\mathcal{C},\mathcal{%
D})}:\mathcal{C}\rightarrow Sh_{(\mathcal{C},\mathcal{D})}$ is the abstract
shape functor, while $I_{(\mathcal{C},\mathcal{D})}:Sh_{(\mathcal{C},%
\mathcal{D})}\rightarrow Sh_{\mathcal{C},\mathcal{D})}^{\ast }$ is induced
by the \textquotedblleft inclusion\textquotedblright\ functor $\underline{I}%
\equiv \underline{I}_{\mathcal{D}}:pro$-$\mathcal{D}\rightarrow pro^{\ast }$-%
$\mathcal{D}$.

As in the case of the abstract shape, the most interesting example of the
above construction is $\mathcal{C}=HTop$ - the homotopy category of
topological spaces and $\mathcal{D}=HPol$ - the homotopy category of
polyhedra (or $\mathcal{D}=HANR$ - the homotopy category of ANR's for metric
spaces). In this case, one speaks about the (ordinary or standard) \emph{%
coarse shape category}

$Sh_{(HTop,HPol)}^{\ast }\equiv Sh^{\ast }(Top)\equiv Sh^{\ast }$

\noindent of topological spaces and of (ordinary or standard) \emph{coarse} 
\emph{shape functor}

$S^{\ast }:HTop\rightarrow Sh^{\ast }$,

\noindent which factorizes as $S^{\ast }=IS$, where $S:HTop\rightarrow Sh$
is the shape functor, and $I:Sh\rightarrow Sh^{\ast }$ is induced by the
\textquotedblleft inclusion\textquotedblright\ functor $\underline{I}\equiv
pro$-$HPol\rightarrow pro^{\ast }$-$HPol$.

The realizing category for $Sh^{\ast }$ is the category $pro^{\ast }$-$HPol$
(or $pro^{\ast }$-$HANR$). The underlying theory might be called the
(ordinary or standard) \emph{coarse shape theory} (for topological spaces).
Clearly, on locally nice spaces ( polyhedra, CW-complexes, ANR's, \ldots )
the coarse shape type classification coincides with the shape type
classification and, consequently, with the homotopy type classification.
However, in general (even for metrizable continua), the shape type
classification is strictly coarser than the homotopy type classification,
and the coarse shape type classification is strictly coarser than the shape
type classification.

\section{Enriched pro-categories}

Given a category $\mathcal{C}$, we are going to construct a class of
categories having the same objects - all inverse systems in the category $%
\mathcal{C}$ - by enriching the morphism sets such that $pro$-$\mathcal{C}$
and $pro^{\ast }$-$\mathcal{C}$ become the very special cases of these new
categories, so called \emph{enriched pro-categories.}

\begin{definition}
\label{D1}Let $\mathcal{C}$ be a category, let $\boldsymbol{X}=\left(
X_{\lambda },p_{\lambda \lambda ^{\prime }},\Lambda \right) $ and $%
\boldsymbol{Y}=(Y_{\mu },q_{\mu \mu ^{\prime }},M)$ be inverse systems in $%
\mathcal{C}$ and let $J=(J,\leq )$ be a directed partially ordered set. A $J$%
-\textbf{morphism }(\textbf{of }$\boldsymbol{X}$ \textbf{to} $\boldsymbol{Y}$%
\textbf{\ in} $\mathcal{C}$) is every triple $(\boldsymbol{X},(f,(f_{\mu
}^{j})),\boldsymbol{Y})$, denoted by $(f,f_{\mu }^{j}):\boldsymbol{X}%
\rightarrow \boldsymbol{Y}$, where $(f,(f_{\mu }^{j}))$ is an ordered pair
consisting of a function $f:M\rightarrow \Lambda $, called the \textbf{index
function}, and, for each $\mu \in M$, of a family $(f_{\mu }^{j})$ of $%
\mathcal{C}$-morphisms $f_{\mu }^{j}:X_{f\left( \mu \right) }\rightarrow
Y_{\mu },$ $j\in J,$ such that, for every related pair $\mu \leq \mu
^{\prime }$ in $M$, there exists a $\lambda \in \Lambda ,$ $\lambda
\geqslant f(\mu ),f\left( \mu ^{\prime }\right) $, and there exists a $j\in
J $ so that, for every $j^{\prime }\geqslant j$,

$f_{\mu }^{j^{\prime }}p_{f(\mu )\lambda }=q_{\mu \mu ^{\prime }}f_{\mu
^{\prime }}^{j^{\prime }}p_{f(\mu ^{\prime })\lambda }$.

\noindent If the index function $f$ is increasing and, for every pair $\mu
\leq \mu ^{\prime }$, one may put $\lambda =f(\mu ^{\prime })$, then $%
(f,f_{\mu }^{j})$ is said to be a \textbf{simple} $J$-morphism. If, in
addition, $M=\Lambda $ and $f=1_{\Lambda }$, then $(1_{\Lambda },f_{\lambda
}^{j})$ is said to be a \textbf{level} $J$-morphism. Further, if the
equality holds for every $j\in J$, then $(f,f_{\mu }^{j}):\boldsymbol{X}%
\rightarrow \boldsymbol{Y}$ is said to be a \textbf{commutative} $J$%
-morphism. (If there exists $\min J\equiv j_{\ast }$, the commutativity
means that one may put $j=j_{\ast }$.)
\end{definition}

\begin{remark}
\label{R1}The equality condition of Definition 1 obviously holds for every $%
\lambda ^{\prime }\geq \lambda $ as well. Every commutative $J$-morphism of
inverse systems $(f,f_{\mu }^{j}):\boldsymbol{X}\rightarrow \boldsymbol{Y}$
yields a family of morphisms $(f^{j}=f,f_{\mu }^{j}):\boldsymbol{X}%
\rightarrow \boldsymbol{Y}$, $j\in J$, of $i\alpha v$-$\mathcal{C}$. On the
other side, every family of simple morphisms $(f^{j},f_{\mu }^{j}):%
\boldsymbol{X}\rightarrow \boldsymbol{Y}$, $j\in J$, of $i\alpha v$-$%
\mathcal{C}$, such that $f^{j}=f$ for all $j$, determines the unique
commutative $J$-morphism of the inverse systems $(f,f_{\mu }^{j}):%
\boldsymbol{X}\rightarrow \boldsymbol{Y}$. This indicates the significant
difference between (a huge generalization of) the standard morphisms of
inverse systems comparing to the new $J$-morphisms.
\end{remark}

\begin{lemma}
\label{L1}Let $\left( f,f_{\mu }^{j}\right) :\boldsymbol{X}\rightarrow 
\boldsymbol{Y}$ and $\left( g,g_{\nu }^{j}\right) :\boldsymbol{Y}\rightarrow 
\boldsymbol{Z}=(Z_{\nu },r_{\nu \nu ^{\prime }},N)$ be $J$-morphisms (of
inverse systems in a category $\mathcal{C}$). Then $\left( h,h_{\nu
}^{j}\right) ,$ where $h=fg$ and $h_{\nu }^{j}=g_{\nu }^{j}f_{g\left( \nu
\right) }^{j},$ $j\in J$, $\nu \in N$, is a $J$-morphism of $\boldsymbol{X}$
to $\boldsymbol{Z}$.
\end{lemma}

\begin{proof}
Let $\nu ,\nu ^{\prime }\in N$, $\nu \leq \nu ^{\prime }$, be given. Since $%
\left( g,g_{\nu }^{j}\right) $ is a $J$-morphism, there exists a $\mu \in M$%
, $\mu \geq g(\nu ),g(\nu ^{\prime })$, and there exists a $j_{0}\in J$ such
that, for every $j^{\prime }\geq j_{0}$,

$g_{\nu }^{j^{\prime }}q_{g(\nu )\mu }=r_{\nu \nu ^{\prime }}g_{\nu ^{\prime
}}^{j^{\prime }}q_{g(\nu ^{\prime })\mu }$.

\noindent Since $\left( f,f_{\mu }^{j}\right) $ is a $J$-morphism, for the
pair $g(\nu )\leq \mu $, there exist a $\lambda _{1}\geq fg(\nu ),f(\mu )$
in $\Lambda $ and a $j_{1}\in J$ such that, for every $j^{\prime }\geq j_{1}$%
,

$f_{g(\nu )}^{j^{\prime }}p_{fg(\nu )\lambda _{1}}=q_{g(\nu )\mu }f_{\mu
}^{j^{\prime }}p_{f(\mu )\lambda _{1}}$.

\noindent Further, for the pair $g(\nu ^{\prime })\leq \mu $, there exist a $%
\lambda _{2}\geq fg(\nu ^{\prime }),f(\mu )$ in $\Lambda $ and a $j_{2}\in J$
such that, for every $j^{\prime }\geq j_{2}$,

$f_{g(\nu ^{\prime \prime })}^{j^{\prime }}p_{fg(\nu ^{\prime })\lambda
_{2}}=q_{g(\nu ^{\prime })\mu }f_{\mu }^{j^{\prime }}p_{f(\mu )\lambda _{2}}$%
.

\noindent Since $\Lambda $ and $J$ are directed, there exist a $\lambda \in
\Lambda $, $\lambda \geq \lambda _{1},\lambda _{2}$, and a $j\in J$, $j\geq
j_{0},j_{1},j_{2},$ respectively. Then, for every $j^{\prime }\geq j$, one
straightforwardly establishes

$g_{\nu }^{j^{\prime }}f_{g(\nu )}^{j^{\prime }}p_{fg(\nu )\lambda }=r_{\nu
\nu ^{\prime }}g_{\nu ^{\prime }}^{j^{\prime }}f_{g(\nu ^{\prime
})}^{j^{\prime }}p_{fg(\nu ^{\prime })\lambda }$,

\noindent which proves that $(h=fg,h_{\nu }^{j}=g_{\nu }^{j}f_{g\left( \nu
\right) }^{j}):\boldsymbol{X}\rightarrow \boldsymbol{Z}$ is a $J$-morphism.
\end{proof}

Lemma 5. enables us to define the \emph{composition} of $J$-morphisms of
inverse systems: If $(f,f_{\mu }^{j}):\boldsymbol{X}\rightarrow \boldsymbol{Y%
}$ and $(g,g_{\nu }^{j}):\boldsymbol{Y}\rightarrow \boldsymbol{Z}$, then $%
(g,g_{\nu }^{j})(f,f_{\mu }^{j})=(h,h_{\nu }^{j}):\boldsymbol{X}\rightarrow 
\boldsymbol{Z}$, where $h=fg$ i $h_{\nu }^{j}=g_{\nu }^{j}f_{g(\nu )}^{j}$.
Clearly, this composition is associative.

\begin{lemma}
\label{L2}The composition of commutative $J$-morphisms of inverse systems in 
$\mathcal{C}$ is a commutative $J$-morphism.
\end{lemma}

\begin{proof}
It is a straightforward consequence of the defining coordinatewise (by $j\in
J$) composition.
\end{proof}

Given an inverse system $\boldsymbol{X}=(X_{\lambda },p_{\lambda \lambda
^{\prime }},\Lambda )$ in $\mathcal{C}$, let $\left( 1_{\Lambda
},1_{X_{\lambda }}^{j}\right) $, consists of the identity function $%
1_{\Lambda }$ and, for each $\lambda \in \Lambda $, of the family induced by
the same identity morphism $1_{X_{\lambda }}^{j}=1_{X_{\lambda }}$, $j\in J$%
, of $\mathcal{C}$. Then $(1_{\Lambda },1_{X_{\lambda }}^{j}):\boldsymbol{X}%
\rightarrow \boldsymbol{X}$ is a $J$-morphism (commutative and leveled). One
readily sees that, for every $(f,f_{\mu }^{j}):\boldsymbol{X}\rightarrow 
\boldsymbol{Y}$ and every $(g,g_{\lambda }^{j}):\boldsymbol{Z}\rightarrow 
\boldsymbol{X}$, $(f,f_{\mu }^{j})(1_{\Lambda },1_{X_{\lambda
}}^{j})=(f,f_{\mu }^{j})$ and $(1_{\Lambda },1_{X_{\lambda
}}^{j})(g,g_{\lambda }^{j})=(g,g_{\lambda }^{j})$ hold. Thus, $(1_{\Lambda
},1_{X_{\lambda }}^{j})$ may be called the \emph{identity} $J$-morphism on $%
\boldsymbol{X}$.

By summarizing, for every category $\mathcal{C}$ and every directed
partially ordered set $J$, there exists a category, denoted by $inv^{J}$-$%
\mathcal{C}$, consisting of the object class $Ob(inv^{J}$-$\mathcal{C}%
)=Ob(inv$-$\mathcal{C})$ and of the morphism class $Mor(inv^{J}$-$\mathcal{C}%
)$ of all the sets $(inv^{J}$-$\mathcal{C})(\boldsymbol{X},\boldsymbol{Y})$
of all $J$-morphisms $(f,f_{\mu }^{n})$ of $\boldsymbol{X}$ to $\boldsymbol{Y%
}$, endowed with the composition and identities described above. By Lemma
2., there exists a subcategory $(inv^{J}$-$\mathcal{C})_{c}$ of $inv^{J}$-$%
\mathcal{C}$ with the same object class and with the morphism class $%
Mor(inv^{J}$-$\mathcal{C})_{c}$ consisting of all commutative $J$-morphisms
of inverse systems in $\mathcal{C}$.

Let us now define an appropriate equivalence relation on each set $(inv^{J}$-%
$\mathcal{C})(\boldsymbol{X},\boldsymbol{Y})$.

\begin{definition}
\label{D2}A $J$-morphism $(f,f_{\mu }^{j}):\boldsymbol{X}\rightarrow 
\boldsymbol{Y}$ of inverse systems in $\mathcal{C}$ is said to be \textbf{%
equivalent to} a $J$-morphism $(f^{\prime },f_{\mu }^{\prime j}):\boldsymbol{%
X}\rightarrow \boldsymbol{Y}$, denoted by $(f,f_{\mu }^{j})\sim (f^{\prime
},f_{\mu }^{\prime j}),if$ every $\mu \in M$ admits a $\lambda \in \Lambda $%
, $\lambda \geqslant f(\mu ),f^{\prime }(\mu )$, and a $j\in J$ such that,
for every $j^{\prime }\geq j$,

$f_{\mu }^{j^{\prime }}p_{f(\mu )\lambda }=f_{\mu }^{\prime j^{\prime
}}p_{f^{\prime }(\mu )\lambda }$.
\end{definition}

\begin{lemma}
\label{L3}The defining equality holds for every $\lambda ^{\prime }\geq
\lambda $ as well, and the relation $\sim $ is an equivalence relation on
each set $(inv^{J}$-$\mathcal{C})(\boldsymbol{X},\boldsymbol{Y}).$ The
equivalence class $[(f,f_{\mu }^{j})]$ of a $J$-morphism $(f,f_{\mu }^{j}):%
\boldsymbol{X}\rightarrow \boldsymbol{Y}$ is briefly denoted by $\boldsymbol{%
f}^{J}\equiv \boldsymbol{f}$.
\end{lemma}

\begin{proof}
The first claim is trivial. The relation $\sim $ is obviously reflexive and
symmetric. To prove transitivity, let, for a given $\mu \in M$, the indices $%
\lambda _{1}$ and $j_{1}$ realize the first relation, $(f,f_{\mu }^{j})\sim
(f^{\prime },f_{\mu }^{\prime j})$, and the indices $\lambda _{2}$ and $%
j_{2} $ - the second one - $(f^{\prime },f_{\mu }^{\prime j})\sim (f^{\prime
\prime },f_{\mu }^{\prime \prime j})$. Since $\Lambda $ and $J$ are
directed, there exist a $\lambda \geq \lambda _{1},\lambda _{2}$ and a $%
j\geq j_{1},j_{2}$ respectively, that realize transitivity, $(f,f_{\mu
}^{j})\sim f^{\prime \prime },f_{\mu }^{\prime \prime j})$.
\end{proof}

\begin{lemma}
\label{L4}Let $(f,f_{\mu }^{j}),(f^{\prime },f_{\mu }^{\prime j}):%
\boldsymbol{X}\rightarrow \boldsymbol{Y}$ and $(g,g_{\nu }^{j}),(g^{\prime
},g_{\nu }^{\prime j}):\boldsymbol{Y}\rightarrow \boldsymbol{Z}$ be $J$%
-morphisms of inverse systems in $\mathcal{C}$. If $(f,f_{\mu }^{j})\sim
(f^{\prime },f_{\mu }^{\prime j})$ and $(g,g_{\nu }^{j})\sim (g^{\prime
},g_{\nu }^{\prime j})$, then $(g,g_{\nu }^{j})(f,f_{\mu }^{j})\sim
(g^{\prime },g_{\nu }^{\prime j})(f^{\prime },f_{\mu }^{\prime j})$.
\end{lemma}

\begin{proof}
According to Lemma 3. (transitivity), it suffices to prove that $(g,g_{\nu
}^{j})(f,f_{\mu }^{j})\sim (g,g_{\nu }^{j})(f^{\prime },f_{\mu }^{\prime j})$
and $(g,g_{\nu }^{j})(f,f_{\mu }^{j})\sim (g^{\prime },g_{\nu }^{\prime
j})(f,f_{\mu }^{j})$. Given a $\nu \in N$, choose a $\lambda \in \Lambda $, $%
\lambda \geq fg(\nu ),f^{\prime }g(\nu )$, and a $j\in J$, by $(f,f_{\mu
}^{j})\sim (f^{\prime },f_{\mu }^{\prime j})$ for $\mu =g(\nu )$. Then, for
every $j^{\prime }\geq j$,

$g_{\nu }^{j^{\prime }}f_{g(\nu )}^{j^{\prime }}p_{fg(\nu )\lambda }=g_{\nu
}^{j^{\prime }}f_{g(\nu )}^{\prime j^{\prime }}p_{f^{\prime }g(\nu )\lambda
} $.

\noindent Thus, $(g,g_{\nu }^{j})(f,f_{\mu }^{j})\sim (g,g_{\nu
}^{j})(f^{\prime },f_{\mu }^{\prime j})$. Further, if $(g,g_{\nu }^{j})\sim
(g^{\prime },g_{\nu }^{\prime j})$, then for a given $\nu \in N$ there exist
a $\mu \geq g(\nu ),g^{\prime }(\nu )$ and a $j_{1}$ such that

$g_{\nu }^{j^{\prime }}q_{g(\nu )\mu }=g_{\nu }^{\prime j^{\prime
}}q_{g^{\prime }(\nu )\mu }$,

\noindent whenever $j^{\prime }\geq j_{1}$. Since $(f,f_{\mu }^{j})$ is a $J$%
-morphism, there exist a $\lambda _{1}\geq fg(\nu ),f(\mu )$ and a $j_{2}$
such that, for every $j^{\prime }\geq j_{2}$,

$f_{g(\nu )}^{j^{\prime }}p_{fg(\nu )\lambda _{1}}=q_{g(\nu )\mu }f_{\mu
}^{j^{\prime }}p_{f(\mu )\lambda _{1}}$.

\noindent In the same way, there exist a $\lambda _{2}\geq fg^{\prime }(\nu
),f(\mu )$ and a $j_{3}$ such that, for every $j^{\prime }\geq j_{3}$,

$f_{g^{\prime }(\nu )}^{j^{\prime }}p_{fg^{\prime }(\nu )\lambda
_{2}}=q_{g^{\prime }(\nu )\mu }f_{\mu }^{j^{\prime }}p_{f(\mu )\lambda _{2}}$%
.

\noindent Since $\Lambda $ and $J$ are directed, there exist a $\lambda \geq
\lambda _{1},\lambda _{2}$ and a $j\geq j_{1},j_{2},j_{3}$ respectively.
Then, for every $j^{\prime }\geq j$,

$g_{\nu }^{j^{\prime }}f_{g(\nu )}^{j^{\prime }}p_{fg(\nu )\lambda }=g_{\nu
}^{\prime j^{\prime }}f_{g^{\prime }(\nu )}^{j^{\prime }}p_{fg^{\prime }(\nu
)\lambda }$.

\noindent Therefore, $(g,g_{\nu }^{j})(f,f_{\mu }^{j})\sim (g^{\prime
},g_{\nu }^{\prime j})(f,f_{\mu }^{j})$.
\end{proof}

By Lemmata 3 and 4 one may compose the equivalence classes of $J$-morphisms
of inverse systems in $\mathcal{C}$ by means of any pair of their
representatives, i.e., $\boldsymbol{gf}=\boldsymbol{h}\equiv \lbrack
(h,h_{\nu }^{j})]$, where $(h,h_{\nu }^{j})=(g,g_{\nu }^{j})(f,f_{\mu
}^{j})=(fg,g_{\nu }^{j}f_{g(\nu )}^{j})$. The corresponding quotient
category $(inv^{J}$-$\mathcal{C})/_{\sim }$ is denoted by $pro^{J}$-$%
\mathcal{C}$. There exists a subcategory $(pro^{J}$-$\mathcal{C}%
)_{c}\subseteq pro^{J}$-$\mathcal{C}$ determined by all equivalence classes
having commutative representatives. Clearly, $(pro^{J}$-$\mathcal{C})_{c}$
is isomorphic to the quotient category $(inv^{J}$-$\mathcal{C})_{c}/_{\sim }$%
. Further, one may consider $pro$-$\mathcal{C}=(inv$-$\mathcal{C})/\sim $ as
a subcategory of $(pro^{J}$-$\mathcal{C})_{c}$ and, consequently, as a
subcategory of $pro^{J}$-$\mathcal{C}$ (see also Theorem 1 below). First,
recall the well known lemma (see [24], Lemma I. 1.1.):

\begin{lemma}
\label{L5}Let $(\Lambda ,\leq )$ be a directed set and let $(M,\leq )$ be a
cofinite directed set. Then every function $f:M\rightarrow \Lambda $ admits
an increasing function $f^{\prime }:M\rightarrow \Lambda $ such that $f\leq
f^{\prime }$.
\end{lemma}

\begin{lemma}
\label{L6}Let $\boldsymbol{X}=(X_{\lambda },p_{\lambda \lambda ^{\prime
}},\Lambda )$ and $\boldsymbol{Y}=(Y_{\mu },q_{\mu \mu ^{\prime }},M)$ be
inverse systems in $\mathcal{C}$ with $M$ cofinite. Then every morphism $%
\boldsymbol{f}=[(f,f_{\mu }^{j})]:\boldsymbol{X}\rightarrow \boldsymbol{Y}$
of $pro^{J}$-$\mathcal{C}$ admits a simple representative $(f^{\prime
},f_{\mu }^{\prime j}):\boldsymbol{X}\rightarrow \boldsymbol{Y}$.
\end{lemma}

\begin{proof}
Let $\mu \in M$. If $\mu $ has no predecessors, choose any $\lambda \in
\Lambda $, $\lambda \geq f(\mu )$, and put $\varphi (\mu )=\lambda $. If $%
\mu $ is not an initial element of $M$, let $\mu _{1},\ldots ,\mu _{m}\in M$%
, $m\in \mathbb{N}$, be all the predecessors of $\mu $ ($M$ is cofinite).
Since $(f,f_{\mu }^{j})$ is a $J$-morphism, for every $i\in \{1,\ldots ,m\}$
and every pair $\mu _{i}\leq \mu $, there exists a $\lambda _{i}\in \Lambda $%
, $\lambda _{i}\geq f(\mu _{i}),f(\mu )$, and there exists a $j_{i}\in J$,
such that, for every $j^{\prime }\geq j_{i}$, the appropriate condition
holds. Choose any $\lambda \in \Lambda $, $\lambda \geq \lambda _{i}$ for
all $i\in \{1,\ldots ,m\}$ ($\Lambda $ is directed), and put $\varphi (\mu
)=\lambda $. This defines a function $\varphi :M\rightarrow \Lambda $.
Notice that $f\leq \varphi $. By Lemma 5., there exists an increasing
function $f^{\prime }:M\rightarrow \Lambda $ such that $\varphi \leq
f^{\prime }$. Hence, $f\leq f^{\prime }$. Now, for every $\mu \in M$, put $%
f_{\mu }^{\prime j}=f_{\mu }^{j}p_{f(\mu )f^{\prime }(\mu )}$. One readily
verifies that $(f^{\prime },f_{\mu }^{\prime j}):\boldsymbol{X}\rightarrow 
\boldsymbol{Y}$ is a simple $J$-morphism and that $(f^{\prime },f_{\mu
}^{\prime j})\sim (f,f_{\mu }^{j})$.
\end{proof}

Let us define a certain functor $\underline{I}\equiv \underline{I}_{\mathcal{%
C}}^{J}:pro$-$\mathcal{C}\rightarrow pro^{J}$-$\mathcal{C}.$ Put $\underline{%
I}\left( \boldsymbol{X}\right) =\boldsymbol{X},$ for every inverse system $%
\boldsymbol{X}$ in $\mathcal{C}$. If $\boldsymbol{f}\in pro$-$\mathcal{C}(%
\boldsymbol{X},\boldsymbol{Y})$ and if $\left( f,f_{\mu }\right) $ is a
representative of $\boldsymbol{f}$, put

$\underline{I}\left( \boldsymbol{f}\right) =\boldsymbol{f}^{J}=[(f,f_{\mu
}^{j})]\in (pro^{J}$-$\mathcal{C})(\boldsymbol{X},\boldsymbol{Y}),$

\noindent where $\left( f,f_{\mu }^{j}\right) $ is \emph{induced by} $\left(
f,f_{\mu }\right) $, i.e., for each $\mu \in M$, $f_{\mu }^{j}=f_{\mu }$ for
all $j\in J$. One straightforwardly verifies that $\underline{I}\left( 
\boldsymbol{f}\right) $ is well defined. Notice that every induced $J$%
-morphism is commutative. Therefore, $\underline{I}$ is a functor of $pro$-$%
\mathcal{C}$ to the subcategory $(pro^{J}$-$\mathcal{C})_{c}\subseteq
pro^{J} $-$\mathcal{C}$.

\begin{theorem}
\label{T1}The functor $\underline{I}:pro$-$\mathcal{C}\rightarrow (pro^{J}$-$%
\mathcal{C})_{c}\subseteq pro^{J}$-$\mathcal{C}$ is faithful.
\end{theorem}

\begin{proof}
The functoriality follows straightforwardly. Let $\boldsymbol{f}^{J}=%
\underline{I}\left( \boldsymbol{f}\right) =\underline{I}(\boldsymbol{f}%
^{\prime })=\boldsymbol{f}^{\prime J}$. Let $\left( f,f_{\mu }\right) $ and $%
\left( f^{\prime },f_{\mu }^{\prime }\right) $ be any representatives of $%
\boldsymbol{f}$ and $\boldsymbol{f}^{\prime }$ respectively. By definition
of the functor $\underline{I}$, $\boldsymbol{f}^{J}=[(f,f_{\mu }^{j}=f_{\mu
})]$ and $\boldsymbol{f}^{\prime J}=[(f^{\prime },f_{\mu }^{\prime j}=f_{\mu
}^{\prime })]$. Since $(f,f_{\mu }^{j})\sim (f^{\prime },f_{\mu }^{\prime
j}) $, for every $\mu \in M$, there exist a $\lambda \geq f(\mu ),f^{\prime
}(\mu )$ and a $j$ such that, for every $j^{\prime }\geq j$,

$f_{\mu }^{j^{\prime }}p_{f(\mu )\lambda }=f_{\mu }^{\prime j^{\prime
}}p_{f^{\prime }(\mu )\lambda }$.

\noindent This means that

$f_{\mu }p_{f(\mu )\lambda }=f_{\mu }^{\prime }p_{f^{\prime }(\mu )\lambda }$

\noindent holds. Therefore, $\left( f,f_{\mu }\right) \sim \left( f^{\prime
},f_{\mu }^{\prime }\right) $, i.e., $\boldsymbol{f}=\boldsymbol{f}^{\prime
} $.
\end{proof}

\begin{remark}
\label{R2}The functor $\underline{I}$ is not full. For instance, let us
consider the restriction $(pro$-$\mathcal{C})\left( \boldsymbol{X},%
\boldsymbol{T}\right) \rightarrow (pro^{J}$-$\mathcal{C})_{c}\left( 
\boldsymbol{X},\boldsymbol{T}\right) $, where $\boldsymbol{T}=(T_{0}\equiv
T) $ is a rudimentary inverse system. Let $\boldsymbol{f}\in (pro$-$\mathcal{%
C})\left( \boldsymbol{X},\boldsymbol{T}\right) $. Then every representative $%
(f,f_{0})$ of $\boldsymbol{f}$ is uniquely determined by a $\lambda _{0}\in
\Lambda $ ($f(0)=\lambda _{0}$) and by a morphism $f_{0}\equiv f_{\lambda
_{0}}\in \mathcal{C}(X_{\lambda },T)$. However, it is not the case for an $%
\boldsymbol{f}^{J}\in (pro^{J}$-$\mathcal{C})_{c}\left( \boldsymbol{X},%
\boldsymbol{T}\right) $. Indeed, if $(f,f_{0}^{j})$ is a representative of $%
\boldsymbol{f}^{J}$, then $f(0)=\lambda _{0}\in \Lambda $, while $%
(f_{0}^{j}\equiv f_{\lambda _{0}}^{j})_{j\in J}$ is a family of morphisms $%
f_{\lambda _{0}}^{j}\in \mathcal{C}(X_{\lambda _{0}},T)$. Notice that $%
(f,f_{0}^{j})\sim (f^{\prime },f_{0}^{\prime j})$ if and only if

$(\exists \lambda \geqslant \lambda _{0},\lambda _{0}^{\prime })\left(
\exists j\right) \left( \forall j^{\prime }\geqslant j\right) $ $%
f_{0}^{j^{\prime }}p_{\lambda _{0}\lambda }=f_{0}^{\prime j^{\prime
}}p_{\lambda _{0}^{\prime }\lambda }$.
\end{remark}

By the well known \textquotedblleft Marde\v{s}i\'{c} trick\textquotedblright
, every inverse system $\boldsymbol{X}$ in $\mathcal{C}$ is isomorphic (in $%
pro$-$\mathcal{C}$) to a cofinite inverse system $\boldsymbol{X}^{\prime }$.
If $\boldsymbol{f}:\boldsymbol{X}\rightarrow \boldsymbol{X}^{\prime }$ is an
isomorphism of $pro$-$\mathcal{C},$ then $\underline{I}\left( \boldsymbol{f}%
\right) :\boldsymbol{X}\rightarrow \boldsymbol{X}^{\prime }$ is an
isomorphism of $pro^{J}$-$\mathcal{C}$. Therefore, the next corollary holds.

\begin{corollary}
\label{C1}Every inverse system $\boldsymbol{X}$ in $\mathcal{C}$ is
isomorphic in $pro^{J}$-$\mathcal{C}$ to a cofinite inverse system $%
\boldsymbol{X}^{\prime }$.
\end{corollary}

A morphism $\boldsymbol{f}:\boldsymbol{X}\rightarrow \boldsymbol{Y}$ of $%
pro^{J}$-$\mathcal{C}$ does not admit, in general, a level representative.
However, the following \textquotedblleft reindexing
theorem\textquotedblright\ will help to overcome some technical difficulties
concerning this fact.

\begin{theorem}
\label{T2}Let $\boldsymbol{f}\in (pro^{J}$-$\mathcal{\mathcal{C}})(%
\boldsymbol{X},\boldsymbol{Y})$. Then there exist inverse systems $%
\boldsymbol{X}^{\prime }$ and $\boldsymbol{Y}^{\prime }in$ $\mathcal{C}$
having the same cofinite index set $(N,\leq )$, there exists a morphism $%
\boldsymbol{f}^{\prime }:\boldsymbol{X}^{\prime }\rightarrow \boldsymbol{Y}%
^{\prime }$ having a level representative $(1_{N},f_{\nu }^{\prime j})$ and
there exist isomorphisms $\boldsymbol{i}:\boldsymbol{X}\rightarrow 
\boldsymbol{X}^{\prime }$ and $\boldsymbol{j}:\boldsymbol{Y}\rightarrow 
\boldsymbol{Y}^{\prime }$ of $pro^{J}$-$\mathcal{C}$, such that the
following diagram in $pro^{J}$-$\mathcal{C}$ commutes

$%
\begin{array}{ccc}
\boldsymbol{X} & \overset{\boldsymbol{f}}{\rightarrow } & \boldsymbol{Y} \\ 
\boldsymbol{i}\downarrow &  & \downarrow \boldsymbol{j} \\ 
\boldsymbol{X}^{\prime } & \overset{\boldsymbol{f}^{\prime }}{\rightarrow }
& \boldsymbol{Y}^{\prime }%
\end{array}%
$.
\end{theorem}

\begin{proof}
Let $\boldsymbol{f}\in (pro^{J}$-$\mathcal{\mathcal{C}})(\boldsymbol{X},%
\boldsymbol{Y})$. By Corollary 1, there exist cofinite inverse systems $%
\widetilde{\boldsymbol{X}}=(\widetilde{X}_{\alpha },\widetilde{p}_{\alpha
\alpha ^{\prime }},A)$ and $\widetilde{\boldsymbol{Y}}=(\widetilde{Y}_{\beta
},\widetilde{q}_{\beta \beta ^{\prime }},B)$, and there exist isomorphisms $%
\boldsymbol{u}:\boldsymbol{X}\rightarrow \widetilde{\boldsymbol{X}}$ and $%
\boldsymbol{v}:\boldsymbol{Y}\rightarrow \widetilde{\boldsymbol{Y}}$ of $%
pro^{J}$-$\mathcal{C}$. Let $\widetilde{\boldsymbol{f}}=\boldsymbol{vfu}%
^{-1}:\widetilde{\boldsymbol{X}}\rightarrow \widetilde{\boldsymbol{Y}}$. By
Lemma 6, there exists a simple representative $(w,w_{\beta }^{j})$ of $%
\widetilde{\boldsymbol{f}}$. Let

$N=\{\nu \equiv (\alpha ,\beta )\mid \alpha \in A,\beta \in B,w(\beta )\leq
\alpha \}\subseteq A\times B$,

\noindent and define $(N,\leq )$ coordinatewise., i.e., $\nu =\left( \alpha
,\beta \right) \leq \left( \alpha ^{\prime },\beta ^{\prime }\right) =\nu
^{\prime }$ if and only if $\alpha \leq \alpha ^{\prime }$ in $A$ and $\beta
\leq \beta ^{\prime }$ in $B$. Clearly, $N$ is preordered. Let any $\nu
=\left( \alpha ,\beta \right) ,\nu ^{\prime }=\left( \alpha ^{\prime },\beta
^{\prime }\right) \in N$ be given. Since $B$ is directed, there exists a $%
\beta _{0}\geq \beta ,\beta ^{\prime }$. Since $A$ is directed, there exists
an $\alpha _{0}\geq \alpha ,\alpha ^{\prime },w(\beta _{0})$. Then $(\alpha
_{0},\beta _{0})\equiv \nu _{0}\in N$ and $\nu _{0}\geq \nu ,\nu ^{\prime }$%
. Thus, $N$ is directed. Further, since $A$ and $B$ are cofinite and since $%
N\subseteq A\times B$ is (pre)ordered coordinatewise, the set $N$ is
cofinite too. Let us now construct desired inverse systems $\boldsymbol{X}%
^{\prime }=\left( X_{\nu }^{\prime },p_{\nu \nu ^{\prime }}^{\prime
},N\right) $ and $\boldsymbol{Y}^{\prime }=\left( Y_{\nu }^{\prime },q_{\nu
\nu ^{\prime }}^{\prime },N\right) $. Given a $\nu =(\alpha ,\beta )\in N$,
put $X_{\nu }^{\prime }=\tilde{X}_{\alpha }$ and $Y_{\nu }^{\prime }=\tilde{Y%
}_{\beta }$. For every related pair $\nu =\left( \alpha ,\beta \right) \leq
\left( \alpha ^{\prime },\beta ^{\prime }\right) =\nu ^{\prime }$ in $N$,
put $p_{\nu \nu ^{\prime }}^{\prime }=\widetilde{p}_{\beta \beta ^{\prime }}$
and $q_{\nu \nu ^{\prime }}^{\prime }=\widetilde{q}_{\gamma \gamma ^{\prime
}}$. Now, for each $\nu =(\alpha ,\beta )\in N$ and every $j\in J$, put $%
f_{\nu }^{\prime j}=w_{\beta }^{j}\widetilde{p}_{w(\beta )j}:X_{\nu
}^{\prime }\rightarrow Y_{\nu }^{\prime }$. Then $(1_{N},f_{\nu }^{\prime
j}):\boldsymbol{X}^{\prime }\rightarrow \boldsymbol{Y}^{\prime }$ is a
simple $J$-morphism. Indeed, if $\nu \leq \nu ^{\prime }$, then $\beta \leq
\beta ^{\prime }$, Since $(w,w_{\beta }^{j})$ is simple, there exists a $%
j\in J$ such that, for every $j^{\prime }\geq j$,

$w_{\beta }^{j^{\prime }}\widetilde{p}_{w(\beta )w(\beta ^{\prime })}=%
\widetilde{q}_{\beta \beta ^{\prime }}w_{\beta ^{\prime }}^{j^{\prime }}$.

\noindent Since $\alpha \geq w(\beta )$, $\alpha ^{\prime }\geq w(\beta
^{\prime })$, $w(\beta ^{\prime })\geq w(\beta )$ and $\alpha ^{\prime }\geq
\alpha $, it implies that

$f_{\nu }^{\prime j^{\prime }}p_{\nu \nu ^{\prime }}^{\prime }=w_{\beta
}^{j^{\prime }}\widetilde{p}_{w(\beta )\alpha }\widetilde{p}_{\alpha \alpha
^{\prime }}=w_{\beta }^{j^{\prime }}\widetilde{p}_{w(\beta )w(\beta ^{\prime
})}\widetilde{p}_{w(\beta ^{\prime })\alpha ^{\prime }}=\widetilde{q}_{\beta
\beta ^{\prime }}w_{\beta ^{\prime }}^{j^{\prime }}\widetilde{p}_{w(\beta
^{\prime })\alpha ^{\prime }}=q_{\nu \nu ^{\prime }}^{\prime }f_{\nu
^{\prime }}^{\prime j^{\prime }}$.

\noindent Let $s:N\rightarrow \Lambda $ be defined by putting $s\left( \nu
\right) =\alpha $, where $\nu =(\alpha ,\beta )$, and let, for each $\nu \in
N$ and every $j\in J$, $s_{\nu }^{j}:\widetilde{X}_{\alpha }\rightarrow
X_{\nu }^{\prime }=\widetilde{X}_{\alpha }$ be the identity $1_{\widetilde{X}%
_{\alpha }}$ of $\mathcal{C}$. In the same way, let $t:N\rightarrow M$ be
defined by putting $t\left( \nu \right) =\beta $, and let, for each $\nu \in
N$ and every $j$, $t_{\nu }^{j}:\widetilde{Y}_{\beta }\rightarrow Y_{\nu
}^{\prime }=\widetilde{Y}_{\beta }$ be the identity $1_{\widetilde{Y}_{\beta
}}$ of $\mathcal{C}$. It is readily seen that $\boldsymbol{s}=[(s,s_{\nu
}^{j})]:\widetilde{\boldsymbol{X}}\rightarrow \boldsymbol{X}^{\prime }$ and $%
\boldsymbol{t}=[(t,t_{\nu }^{j})]:\widetilde{\boldsymbol{Y}}\rightarrow 
\boldsymbol{Y}^{\prime }$ are simple commutative morphisms of $pro^{J}$-$%
\mathcal{C}$. Even more, they are induced by morphisms $(s,s_{\nu }=1_{%
\widetilde{X}_{\alpha }})$ and $(t,t_{\nu }=1_{\widetilde{Y}_{\beta }})$ of $%
inv$-$\mathcal{C}$ respectively. Notice that, in $pro$-$\mathcal{C}$, $%
[(s,s_{\nu })]:\widetilde{\boldsymbol{X}}\rightarrow \boldsymbol{X}^{\prime
} $ and $[(t,t_{\nu })]:\widetilde{\boldsymbol{Y}}\rightarrow \boldsymbol{Y}%
^{\prime }$ are isomorphisms. Since $\boldsymbol{s}=\underline{I}([(s,s_{\nu
})])$ and $\boldsymbol{t}=\underline{I}([(t,t_{\nu })])$, we infer that $%
\boldsymbol{s}$ and $\boldsymbol{t}$ are isomorphisms of $pro^{J}$-$\mathcal{%
C}$. Moreover, for every $\nu =(\alpha ,\beta )\in N$ and every $j\in J$,

$t_{\nu }^{j}w_{t\left( \nu \right) }^{j}\widetilde{p}_{wt\left( \nu \right)
\alpha }=w_{\beta }^{j}\widetilde{p}_{w\left( \beta \right) \alpha }=f_{\nu
}^{\prime j}=f_{\nu }^{\prime j}s_{\nu }^{j}$,

\noindent which implies that

$(t,t_{\nu }^{j})(w,w_{\beta }^{j})\sim (1_{N},f_{\nu }^{\prime j})(s,s_{\nu
}^{j})$.

\noindent Therefore, $\boldsymbol{t}\widetilde{\boldsymbol{f}}=\boldsymbol{f}%
^{\prime }\boldsymbol{s}$. Finally, put $\boldsymbol{i}\equiv \boldsymbol{su}%
:\boldsymbol{X}\rightarrow \boldsymbol{X}^{\prime }$ and $\boldsymbol{j}%
\equiv \boldsymbol{tv}:\boldsymbol{Y}\rightarrow \boldsymbol{Y}^{\prime }$,
which are isomorphisms of $pro^{J}$-$\mathcal{C}$. Then

$\boldsymbol{jf}=\boldsymbol{tvf}=\boldsymbol{t}\widetilde{\boldsymbol{f}}%
\boldsymbol{u}=\boldsymbol{f}^{\prime }\boldsymbol{su}=\boldsymbol{f}%
^{\prime }\boldsymbol{i}$,

\noindent that completes the proof of the theorem.
\end{proof}

\begin{theorem}
\label{T3}Let $\mathcal{C}$ be a category. Then

(i) $pro$-$\mathcal{C}=pro^{(1)}$-$\mathcal{C}$;

(ii) $pro^{\ast }$-$\mathcal{C}=pro^{\mathbb{N}}$-$\mathcal{C}$;

(iii) If $J$ is a directed partially ordered set having $\max J$, then $%
pro^{J}$-$\mathcal{C}\cong pro$-$\mathcal{C}$;

(iv) If $J$ and $K$ are finite directed partially ordered sets, then one may
identify $pro^{J}$-$\mathcal{C}\cong pro^{K}$-$\mathcal{C}\cong pro$-$%
\mathcal{C}$.

(v) If there exists $\max J$, then, for every $L,$ there exists the
canonical inclusion functor $\underline{I}:pro^{J}$-$\mathcal{C}\rightarrow
pro^{L}$-$\mathcal{C}$ keeping the objects fixed.
\end{theorem}

\begin{proof}
Statements (i) and (ii) are obviously true by the definition of $pro^{J}$-$%
\mathcal{C}$. In order to prove (iii), it suffices to show that every

$\boldsymbol{f}=[(f,f_{\mu }^{j})]:\boldsymbol{X}=\left( X_{\lambda
},p_{\lambda \lambda ^{\prime }},\Lambda \right) \rightarrow (Y_{\mu
},q_{\mu \mu ^{\prime }},M)=\boldsymbol{Y}$

\noindent of $pro^{J}$-$\mathcal{C}$ is fully and uniquely determined by

$(f,f_{\mu }^{j^{\ast }})]:\boldsymbol{X}\rightarrow \boldsymbol{Y}$, $%
j^{\ast }\equiv \max J$,

\noindent which belongs to $(inv$-$\mathcal{C})(\boldsymbol{X},\boldsymbol{Y}%
)$. Indeed, since $\max J\equiv j^{\ast }$ exists. Definition 1 implies that

$(\forall \mu \leq \mu ^{\prime })(\exists \lambda \geq f(\mu )\lambda
,f(\mu ^{\prime })$

$f_{\mu }^{j^{\ast }}p_{f(\mu )\lambda }=q_{\mu \mu ^{\prime }}f_{\mu
^{\prime }}^{j^{\ast }}p_{f(\mu ^{\prime })\lambda }$.

\noindent This means that

$(f,f_{\mu }^{j^{\ast }}):\boldsymbol{X}\rightarrow \boldsymbol{Y}$

\noindent is a morphism of $inv$-$\mathcal{C}$. Further, if

$(f^{\prime },f_{\mu }^{\prime j}):\boldsymbol{X}\rightarrow \boldsymbol{Y}$

\noindent is an arbitrary representative of $\boldsymbol{f}$, then

$(f^{\prime },f_{\mu }^{\prime j^{\ast }})]:\boldsymbol{X}\rightarrow 
\boldsymbol{Y}$

\noindent belongs to $(inv$-$\mathcal{C})(\boldsymbol{X},\boldsymbol{Y})$ as
well and, moreover, $(f^{\prime },f_{\mu }^{\prime j^{\ast }})\sim (f,f_{\mu
}^{j^{\ast }})$ in $inv$-$\mathcal{C}$ is equivalent to $(f^{\prime },f_{\mu
}^{\prime j})\sim (f,f_{\mu }^{j})$ in $inv^{J}$-$\mathcal{C}$. The
conclusion follows. Statement (iv) in an immediate consequence of (iii)
because every such finite set must have a unique maximal element. Statement
(v) follows by (iv) because every $\boldsymbol{f}=[(f,f_{\mu }^{j})]\in
(pro^{J}$-$\mathcal{C})(\boldsymbol{X},\boldsymbol{Y})$ is determined by $%
(f,f_{\mu }^{\max J})\in (inv$-$\mathcal{C})(\boldsymbol{X},\boldsymbol{Y})$%
, which induces a unique $\boldsymbol{f}^{\prime }=[(f^{\prime }=f,f_{\mu
}^{\prime l}=f_{\mu }^{\max J})]\in (pro^{L}$-$\mathcal{C})_{c}(\boldsymbol{X%
},\boldsymbol{Y})\subseteq (pro^{L}$-$\mathcal{C})(\boldsymbol{X},%
\boldsymbol{Y})$.
\end{proof}

According to Theorem 3, only a $(J,\leq )$ having no maximal element is
interesting because the existence of $\max J$ turns us back to the
\textquotedblleft trivial\textquotedblright\ case of $pro$-$\mathcal{C}$. In
order to relate $pro^{J}$-$\mathcal{C}$ to a $pro^{K}$-$\mathcal{C}$ in a
\textquotedblleft nontrivial\textquotedblright\ case, we have established
the following fact only.

\begin{theorem}
\label{T4}Let $\mathcal{C}$ be a category, let $J$ be a well ordered set and
let $K$ be a directed partially ordered set, both without maximal elements.
If there exists an increasing function $\phi :J\rightarrow K$ such that $%
\phi \lbrack J]$ is cofinal in $K$, then there exists a functor

$\underline{T}:pro^{J}$-$\mathcal{C}\rightarrow pro^{K}$-$\mathcal{C}$

\noindent keeping the objects fixed, and $T$ does not depend on $\phi $.
Furthermore, for every pair $\boldsymbol{X}$, $\boldsymbol{Y}$ of inverse
systems in $\mathcal{C}$, the equivalence

$(\boldsymbol{X}\cong \boldsymbol{Y}$ in $pro^{J}$-$\mathcal{C}%
)\Leftrightarrow (\boldsymbol{X}\cong \boldsymbol{Y}$ in $pro^{K}$-$\mathcal{%
C})$

\noindent holds true.
\end{theorem}

\begin{proof}
Since $\phi :J\rightarrow K$ is cofinal, for each $k\in K$, the subset

$J_{k}=\{j\mid k\leq \phi (j)\}\subseteq J$

\noindent is not empty. Since $J$ is well ordered, there exists $\min J_{k}$%
. Furthermore,

$k\leq k^{\prime }\Rightarrow j_{k}\equiv \min J_{k}\leq \min J_{k^{\prime
}}\equiv j_{k^{\prime }}$

\noindent because $\phi $ is increasing. Given an

$\boldsymbol{f}=[(f,f_{\mu }^{j})]:\boldsymbol{X}=\left( X_{\lambda
},p_{\lambda \lambda ^{\prime }},\Lambda \right) \rightarrow (Y_{\mu
},q_{\mu \mu ^{\prime }},M)=\boldsymbol{Y}$

\noindent of $pro^{J}$-$\mathcal{C}$, put

$f^{\prime }=f:M\rightarrow \Lambda $ \quad and

$(\forall \mu \in M)(\forall k\in K)$ $f_{\mu }^{\prime k}=f_{\mu
}^{j_{k}}:X_{f^{\prime }(\mu )}\rightarrow Y_{\mu }$.

\noindent Then

$(f^{\prime },f_{\mu }^{\prime k}):\boldsymbol{X}\rightarrow \boldsymbol{Y}$

\noindent is a morphism of $inv^{K}$-$\mathcal{C}$. Indeed, since $(f,f_{\mu
}^{j})$ is a morphism of $inv^{J}$-$\mathcal{C}$, given a related pair $\mu
\leq \mu ^{\prime }$, there exist a $\lambda \geq f(\mu ),f(\mu ^{\prime })$
and a $j$ such that, for every $j^{\prime }\geq j$,

$f_{\mu }^{j^{\prime }}p_{f(\mu )\lambda }=q_{\mu \mu ^{\prime }}f_{\mu
^{\prime }}^{j^{\prime }}p_{f(\mu ^{\prime })\lambda }$.

\noindent Choose $k=\phi (j)$, and let $k^{\prime }\geq k$. Then $%
j_{k^{\prime }}\geq j_{k}$ and

$f_{\mu }^{\prime k^{\prime }}p_{f^{\prime }(\mu )\lambda }=f_{\mu
}^{j_{k^{\prime }}}p_{f(\mu )\lambda }=q_{\mu \mu ^{\prime }}f_{\mu ^{\prime
}}^{j_{k^{\prime }}}p_{f(\mu ^{\prime })\lambda }=$ $q_{\mu \mu ^{\prime
}}f_{\mu ^{\prime }}^{\prime k^{\prime }}p_{f(\mu ^{\prime })\lambda }$.

\noindent that proves the claim. Denote

$\boldsymbol{f}^{\prime }=[(f^{\prime },f_{\mu }^{\prime k})]:\boldsymbol{X}%
\rightarrow \boldsymbol{Y}$

\noindent which is a morphism of $pro^{K}$-$\mathcal{C}$. Now a
straightforward verification shows that the assignments

$\boldsymbol{X}\mapsto \underline{T}(\boldsymbol{X})=\boldsymbol{X}$, $%
\boldsymbol{f}\mapsto \underline{T}(\boldsymbol{f})=\boldsymbol{f}^{\prime }$

\noindent define a functor

$\underline{T}:pro^{J}$-$\mathcal{C}\rightarrow pro^{K}$-$\mathcal{C}$

\noindent Finally, if $\psi :J\rightarrow K$ has the same properties as $%
\phi $, then one readily sees that $(f^{\prime \prime },f_{\mu }^{\prime
\prime k}):\boldsymbol{X}\rightarrow \boldsymbol{Y}$, constructed by means
of $\psi $, is equivalent to $(f^{\prime },f_{\mu }^{\prime k})$ in $inv^{K}$%
-$\mathcal{C}$. Thus, $\underline{T}$ does not depend on any such particular
function. In order to prove the last statement, firstly notice that the
implication

$(\boldsymbol{X}\cong \boldsymbol{Y}$ in $pro^{J}$-$\mathcal{C})\Rightarrow (%
\boldsymbol{X}\cong \boldsymbol{Y}$ in $pro^{K}$-$\mathcal{C})$

\noindent holds because there exists the functor $\underline{T}:pro^{J}$-$%
\mathcal{C}\rightarrow pro^{K}$-$\mathcal{C}$. Conversely, let $\boldsymbol{X%
}\cong \boldsymbol{Y}$ in $pro^{K}$-$\mathcal{C}$. Choose any isomorphism $%
\boldsymbol{g}:\boldsymbol{X}\rightarrow \boldsymbol{Y}$ of $pro^{K}$-$%
\mathcal{C}$, and let $(g,g_{\mu }^{k}):\boldsymbol{X}\rightarrow 
\boldsymbol{Y}$ of $inv^{K}$-$\mathcal{C}$ be any representative of $%
\boldsymbol{g}$. Let us define

$f=g:M\rightarrow \Lambda $ \quad and

$(\forall \mu \in M)(\forall j\in J)$ $f_{\mu }^{j}=g_{\mu }^{\phi
(j)}:X_{f(\mu )}\rightarrow Y_{\mu }$.

\noindent Since $\phi $ is cofinal (i.e., for every $k\in K$ there exists a $%
j\in J$ such that $\phi (j)\geq k$) and increasing (especially, for every $%
j^{\prime }\geq j$, $\phi (j^{\prime }\equiv k^{\prime }\geq \phi (j)\geq k$%
), one can easy verify that

$(f,f_{\mu }^{j}):\boldsymbol{X}\rightarrow \boldsymbol{Y}$

\noindent is a morphism of $inv^{J}$-$\mathcal{C}$, and thus, the
equivalence class

$\boldsymbol{f}=[(f,f_{\mu }^{j})]:\boldsymbol{X}\rightarrow \boldsymbol{Y}$

\noindent is a morphism of $pro^{J}$-$\mathcal{C}$. Let $\boldsymbol{v}%
\equiv \boldsymbol{g}^{-1}:\boldsymbol{Y}\rightarrow \boldsymbol{X}$ of $%
pro^{K}$-$\mathcal{C}$ be the inverse of $\boldsymbol{g}$, and let $%
(v,v_{\lambda }^{k}):\boldsymbol{Y}\rightarrow \boldsymbol{X}$ of $inv^{K}$-$%
\mathcal{C}$ be any representative of $\boldsymbol{v}$. Let us define

$u=v:\Lambda \rightarrow M$ \quad and

$(\forall \lambda \in \Lambda )(\forall j\in J)$ $u_{\lambda
}^{j}=v_{\lambda }^{\phi (j)}:Y_{u(\lambda )}\rightarrow X_{\lambda }$.

\noindent Now, as for $(f,f_{\mu }^{j})$ before, one readily verifies that

$(u,u_{\lambda }^{j}):\boldsymbol{Y}\rightarrow \boldsymbol{X}$

\noindent is a morphism of $inv^{J}$-$\mathcal{C}$, and thus, the
equivalence class

$\boldsymbol{u}=[(u,u_{\lambda }^{j})]:\boldsymbol{Y}\rightarrow \boldsymbol{%
X}$

\noindent is a morphism of $pro^{J}$-$\mathcal{C}$. Since $\boldsymbol{vg}%
=1_{\boldsymbol{X}}$ and $\boldsymbol{gv}=1_{\boldsymbol{Y}}$ in $pro^{K}$-$%
\mathcal{C}$, the relations

$(gv,v_{\lambda }^{k}g_{v(\lambda )}^{k})\sim (1_{\Lambda },1_{\lambda
}^{k}):\boldsymbol{X}\rightarrow \boldsymbol{X}$ \quad and

$(vg,g_{\mu }^{k}v_{g(\mu )}^{k})\sim (1_{M},1_{\mu }^{k});\boldsymbol{Y}%
\rightarrow \boldsymbol{Y}$

\noindent hold in $inv^{K}$-$\mathcal{C}$. Then, by our construction, one
straightforwardly verifies that

$(fu,u_{\lambda }^{j}f_{u(\lambda )}^{j})\sim (1_{\Lambda },1_{\lambda
}^{j}):\boldsymbol{X}\rightarrow \boldsymbol{X}$ \quad and

$(uf,f_{\mu }^{j}u_{f(\mu )}^{j})\sim (1_{M},1_{\mu }^{j});\boldsymbol{Y}%
\rightarrow \boldsymbol{Y}$

\noindent hold in $inv^{J}$-$\mathcal{C}$. Therefore, $\boldsymbol{u}=%
\boldsymbol{f}^{-1}$ is the inverse of $\boldsymbol{f}$ in $pro^{J}$-$%
\mathcal{C}$, implying that $\boldsymbol{X}\cong \boldsymbol{Y}$ in $pro^{J}$%
-$\mathcal{C}$.
\end{proof}

\section{The $J$-shape category of a category}

An enriched pro-category $pro^{J}$-$\mathcal{C}$ is interesting and useful
by itself because, in general, it divides (classifies) the objects into
larger classes (isomorphisms types) than the underlying pro-category $pro$-$%
\mathcal{C}$ (see Examples 7.1 and 7.2 of [19]). Moreover, in many important
cases one can go on much further, i.e., to develop the corresponding $J$%
-shape theory.

Let $\mathcal{D}$ be a full (not essential, but a convenient condition) and
pro-reflective subcategory of $\mathcal{C}$. Let $\boldsymbol{p}%
:X\rightarrow \boldsymbol{X}$ and $\boldsymbol{p}^{\prime }:X\rightarrow 
\boldsymbol{X}^{\prime }$ be $\mathcal{D}$-expansions of the same object $X$
of $\mathcal{C}$, and let $\boldsymbol{q}:Y\rightarrow \boldsymbol{Y}$ and $%
\boldsymbol{q}^{\prime }:Y\rightarrow \boldsymbol{Y}^{\prime }$ be $\mathcal{%
D}$-expansions of the same object $Y$ of $\mathcal{C}$. Then there exist two
canonical isomorphisms $\boldsymbol{i}:\boldsymbol{X}\rightarrow \boldsymbol{%
X}^{\prime }$ and $\boldsymbol{j}:\boldsymbol{Y}\rightarrow \boldsymbol{Y}%
^{\prime }$ of $pro$-$\mathcal{D}$. Consequently, for every directed
partially ordered set $J$, the (induced) morphisms $\boldsymbol{i}\equiv 
\underline{I}(\boldsymbol{i}):\boldsymbol{X}\rightarrow \boldsymbol{X}%
^{\prime }$ and $\boldsymbol{j}\equiv \underline{I}(\boldsymbol{j}):%
\boldsymbol{Y}\rightarrow \boldsymbol{Y}^{\prime }$ are isomorphisms of $%
pro^{J}$-$\mathcal{D}$. A $J$-morphism $\boldsymbol{f}:\boldsymbol{X}%
\rightarrow \boldsymbol{Y}$ is said to be $pro^{J}$-$\mathcal{D}$ \emph{%
equivalent to} a morphism $\boldsymbol{f}^{\prime }:\boldsymbol{X}^{\prime
}\rightarrow \boldsymbol{Y}^{\prime }$, denoted by $\boldsymbol{f}\sim 
\boldsymbol{f}^{\prime }$, if the following diagram in $pro^{J}$-$\mathcal{D}
$ commutes:

$%
\begin{array}{ccc}
\boldsymbol{X} & \overset{\boldsymbol{i}}{\longrightarrow } & \boldsymbol{X}%
^{\prime } \\ 
\boldsymbol{f}\downarrow &  & \downarrow \boldsymbol{f}^{\prime } \\ 
\boldsymbol{Y} & \overset{\boldsymbol{j}}{\longrightarrow } & \boldsymbol{Y}%
^{\prime }%
\end{array}%
$.

\noindent According to the analogous facts in $pro$-$\mathcal{D}$, and since 
$\underline{I}$ is a functor, the diagram defines an equivalence relation on
the appropriate subclass of $Mor(pro^{J}$-$\mathcal{D})$, such that $%
\boldsymbol{f}\sim \boldsymbol{f}^{\prime }$ and $\boldsymbol{g}\sim 
\boldsymbol{g}^{\prime }$ imply $\boldsymbol{gf}\sim \boldsymbol{g}^{\prime }%
\boldsymbol{f}^{\prime }$ whenever these compositions exist. The equivalence
class of such an $\boldsymbol{f}$ is denoted by $\left\langle \boldsymbol{f}%
\right\rangle $. Further, given $\boldsymbol{p}$, $\boldsymbol{p}^{\prime }$%
, $\boldsymbol{q}$, $\boldsymbol{q}^{\prime }$ and $\boldsymbol{f}$, there
exists a unique $\boldsymbol{f}^{\prime }$ ($=\boldsymbol{jfi}^{-1}$) such
that $\boldsymbol{f}\sim \boldsymbol{f}^{\prime }$.

We are now to define the (\emph{abstract}) $J$\emph{-shape category} $Sh_{(%
\mathcal{C},\mathcal{D})}^{J}$ \emph{for} $(\mathcal{C},\mathcal{D})$ as
follows. The objects of $Sh_{(\mathcal{C},\mathcal{D})}^{J}$ are all the
objects of $\mathcal{C}$. A morphism $F\in Sh_{(\mathcal{C},\mathcal{D}%
)}^{J}(X,Y)$ is the $(pro^{J}$-$\mathcal{D}$)-equivalence class $%
\left\langle \boldsymbol{f}\right\rangle $ of a $J$-morphism $\boldsymbol{f}:%
\boldsymbol{X}\rightarrow \boldsymbol{Y}$ of $pro^{J}$-$\mathcal{D}$, with
respect to any choice of a pair of $\mathcal{D}$-expansions $\boldsymbol{p}%
:X\rightarrow \boldsymbol{X}$, $\boldsymbol{q}:Y\rightarrow \boldsymbol{Y}$.
In other words, a $\emph{J}$\emph{-shape morphism} $F:X\rightarrow Y$ is
given by a diagram

$%
\begin{array}{ccc}
\boldsymbol{X} & \overset{\boldsymbol{p}}{\longleftarrow } & X \\ 
\boldsymbol{f}\downarrow & F &  \\ 
\boldsymbol{Y} & \overset{\boldsymbol{q}}{\longleftarrow } & Y%
\end{array}%
$.

\noindent in $pro^{J}$-$\mathcal{C}$. The \emph{composition} of such an $%
F:X\rightarrow Y$, $F=\left\langle \boldsymbol{f}\right\rangle $ and a $%
G:Y\rightarrow Z$, $G=\left\langle \boldsymbol{g}\right\rangle $, is defined
by the representatives, i.e. $GF:X\rightarrow Z$, $GF=\left\langle 
\boldsymbol{gf}\right\rangle $. The\emph{\ identity }$J$\emph{-shape morphism%
} on an object $X$, $1_{X}:X\rightarrow X$, is the $(pro^{J}$-$\mathcal{D})$%
-equivalence class $\left\langle \boldsymbol{1}_{\boldsymbol{X}%
}\right\rangle $ of the identity morphism $\boldsymbol{1}_{\boldsymbol{X}}$
of $pro^{J}$-$\mathcal{D}$. Since

$Sh_{(\mathcal{C},\mathcal{D})}^{J}(X,Y)\approx pro^{J}$-$\mathcal{D}(%
\boldsymbol{X},\boldsymbol{Y})$

\noindent is a set, the $J$-shape category $Sh_{(\mathcal{C},\mathcal{D}%
)}^{J}$ is well defined. One may say that $pro^{J}$-$\mathcal{D}$ is the 
\emph{realizing} category for the $J$-shape category $Sh_{(\mathcal{C},%
\mathcal{D})}^{J}$.

For every $f:X\rightarrow Y$ of $\mathcal{C}$ and every pair of $\mathcal{D}$%
-expansions $\boldsymbol{p}:X\rightarrow \boldsymbol{X}$, $\boldsymbol{q}%
:Y\rightarrow \boldsymbol{Y}$, there exists an $\boldsymbol{f}:\boldsymbol{X}%
\rightarrow \boldsymbol{Y}$ of $pro^{J}$-$\mathcal{D}$, such that the
following diagram in $pro^{J}$-$\mathcal{C}$ commutes:

$%
\begin{array}{ccc}
\boldsymbol{X} & \overset{\boldsymbol{p}}{\longleftarrow } & X \\ 
\boldsymbol{f}\downarrow &  & \downarrow f \\ 
\boldsymbol{Y} & \overset{\boldsymbol{q}}{\longleftarrow } & Y%
\end{array}%
$.

\noindent (Hereby, we consider $\mathcal{C}\subseteq pro$-$\mathcal{C}$ to
be subcategories of $pro^{J}$-$\mathcal{C}$!) The same $f$ and another pair
of $\mathcal{D}$-expansions $\boldsymbol{p}^{\prime }:X\rightarrow 
\boldsymbol{X}^{\prime }$, $\boldsymbol{q}^{\prime }:Y\rightarrow 
\boldsymbol{Y}^{\prime }$ yield an $\boldsymbol{f}^{\prime }:\boldsymbol{X}%
^{\prime }\rightarrow \boldsymbol{Y}^{\prime }$ of $pro^{J}$-$\mathcal{D}$.
Then, however, $\boldsymbol{f}\sim \boldsymbol{f}^{\prime }$ in $pro^{J}$-$%
\mathcal{D}$ must hold. Thus, every morphism $f\in \mathcal{C}(X,Y)$ yields
a $(pro^{J}$-$\mathcal{D})$-equivalence class $\left\langle \boldsymbol{f}%
\right\rangle $, i.e. a $J$-shape morphism $F\in Sh_{(\mathcal{C},\mathcal{D}%
)}^{J}(X,Y)$. If one defines $S^{J}(X)=X$, $X\in Ob\mathcal{C}$, and $%
S^{J}(f)=F=\left\langle \boldsymbol{f}\right\rangle $, $f\in Mor\mathcal{C}$%
, then

$S^{J}\equiv S_{(\mathcal{C},\mathcal{D})}^{J}:\mathcal{C}\rightarrow Sh_{(%
\mathcal{C},\mathcal{D})}^{J}$

\noindent becomes a functor, called the \emph{(abstract)} $J$-\emph{shape
functor}. Comparing to the (abstract) shape functor, we know that the
restriction of $S^{J}$ to $\mathcal{D}$ into the full subcategory of $Sh_{(%
\mathcal{C},\mathcal{D})}^{J}$, determined by $Ob\mathcal{D}$, is \emph{not}
a category isomorphism (Example 3 of [19]). Nevertheless, we shall prove
that $P$ and $Q$ are isomorphic objects of $\mathcal{D}$ if and only if they
are isomorphic in $Sh_{(\mathcal{C},\mathcal{D})}^{J}$, i.e. they are of the
same $J$-shape (Theorem 5 below). Thus, clearly, the $J$-shape type
classification on $\mathcal{D}$ coincides with the shape type
classification. Further, recall that for every $X\in Ob\mathcal{C}$ and
every $Q\in Ob\mathcal{D}$, the shape functor induces a bijection

$S|\cdot :\mathcal{C}(X,Q)\rightarrow Sh_{(\mathcal{C},\mathcal{D})}(X,Q)$.

\noindent However, in the same circumstances, the $J$-shape functor induces
an injection

$S^{J}|\cdot :\mathcal{C}(X,Q)\rightarrow Sh_{(\mathcal{C},\mathcal{D}%
)}^{J}(X,Q)$,

\noindent which, in general, is \emph{not} a surjection (Example 3 of [19]).
Finally, the functor $S_{(\mathcal{C},\mathcal{D})}^{J}$ factorizes as $S_{(%
\mathcal{C},\mathcal{D})}^{J}=I_{(\mathcal{C},\mathcal{D})}S_{(\mathcal{C},%
\mathcal{D})}$, where $S_{(\mathcal{C},\mathcal{D})}:\mathcal{C}\rightarrow
Sh_{(\mathcal{C},\mathcal{D})}$ is the (abstract) shape functor, while $I_{(%
\mathcal{C},\mathcal{D})}:Sh_{(\mathcal{C},\mathcal{D})}\rightarrow Sh_{%
\mathcal{C},\mathcal{D})}^{J}$ is induced by the \textquotedblleft
inclusion\textquotedblright\ functor $\underline{I}\equiv \underline{I}_{%
\mathcal{D}}:pro$-$\mathcal{D}\rightarrow pro^{J}$-$\mathcal{D}$. (This
implies that the induced function $\mathcal{C}(X,Q)\rightarrow Sh_{(\mathcal{%
C},\mathcal{D})}^{J}(X,Q)$ is an injection.)

As in the case of the abstract shape, the most interesting example of the
above construction is $\mathcal{C}=HTop$ - the homotopy category of
topological spaces and $\mathcal{D}=HPol$ - the homotopy category of
polyhedra (or $\mathcal{D}=HANR$ - the homotopy category of ANR's for metric
spaces. In this case, one speaks about the (ordinary or standard) $J$\emph{%
-shape category }

$Sh_{(HTop,HPol)}^{J}\equiv Sh^{J}(Top)\equiv Sh^{J}$

\noindent of topological spaces and of (ordinary or standard) $J$-\emph{%
shape functor}

$S^{J}:HTop\rightarrow Sh^{J}$,

\noindent which factorizes as $S^{J}=IS$, where $S:HTop\rightarrow Sh$ is
the shape functor, and $I:Sh\rightarrow Sh^{J}$ is induced by the
\textquotedblleft inclusion\textquotedblright\ functor $\underline{I}\equiv
pro$-$HPol\rightarrow pro^{J}$-$HPol$. The realizing category for $Sh^{J}$
is the category $pro^{J}$-$HPol$ (or $pro^{J}$-$HANR$). The underlying
theory might be called the (ordinary or standard) $J$\emph{-shape theory}
(for topological spaces). Clearly, on locally nice spaces ( polyhedra,
CW-complexes, ANR's, \ldots ) the $J$-shape type classification coincides
with the shape type classification and, consequently, with the homotopy type
classification.

Similarly to the case of the shape of compacta, let us consider the homotopy
(sub)category of compact metric spaces, $HcM\subseteq HTop$. Since $%
HcPol\subseteq HcM$ and $HcANR\subseteq HcM$ are \textquotedblleft \emph{%
sequentially}\textquotedblright\ pro-reflective (and homotopically
equivalent) subcategories, there exist the $J$-shape category of compacta,

$Sh^{J}(cM)\equiv Sh_{(HcM,HcPol)}^{J}\cong Sh_{(HcM,HcANR)}^{J}$, \noindent

\noindent and the corresponding (restriction of the) $J$-shape functor

$S^{J}:HcM\rightarrow Sh^{J}(cM)$,

\noindent such that $S^{J}=IS$, where $S:HcM\rightarrow Sh(cM)$ is the shape
functor on compacta, and $I:Sh(cM)\rightarrow Sh^{J}(cM)$ is induced by the
\textquotedblleft inclusion\textquotedblright\ functor $\underline{I}:tow$-$%
HcPol\rightarrow tow^{J}$-$HcPol$ (or $\underline{I}:tow$-$HcANR\rightarrow
tow^{J}$-$HcANR$). The category $Sh^{J}(cM)$ is a full subcategory of $%
Sh^{J} $. Notice that the realizing category for $Sh^{J}(cM)$ is the
category $tow^{J}$-$HcPol$ as well as the category $tow^{J}$-$HcANR$.

The following facts are immediate consequences of Theorems 3 and 4 of the
previous section.

\begin{corollary}
\label{C2}Let $\mathcal{C}$ be a category and let $\mathcal{D}\subseteq 
\mathcal{C}$ be a pro-reflective subcategory. Then

(i) $Sh_{(\mathcal{C},\mathcal{D})}=Sh_{(\mathcal{C},\mathcal{D})}^{\{1\}}$;

(ii) $Sh_{(\mathcal{C},\mathcal{D})}^{\ast }=Sh_{(\mathcal{C},\mathcal{D})}^{%
\mathbb{N}}$;

(iii) If $J$ is a directed partially ordered set having $\max J$, then $Sh_{(%
\mathcal{C},\mathcal{D})}^{J}\cong Sh_{(\mathcal{C},\mathcal{D})}$.
\end{corollary}

\begin{corollary}
\label{C3}Let $\mathcal{C}$ be a category, let $\mathcal{D}\subseteq 
\mathcal{C}$ be a pro-reflective subcategory, let $J$ be a well ordered set
and let $K$ be a partially ordered set, both without maximal elements. If
there exists an increasing function $\phi :J\rightarrow K$ such that $\phi
\lbrack J]$ is cofinal in $K$, then there exists a functor

$T:Sh_{(\mathcal{C},\mathcal{D})}^{J}\rightarrow Sh_{(\mathcal{C},\mathcal{D}%
)}^{K}$

\noindent keeping the objects fixed, and $T$ does not depend on $\phi $.
Furthermore, for every pair $X$, $Y$ of objects of $\mathcal{C}$, the
equivalence

$(X\cong Y$ in $Sh_{(\mathcal{C},\mathcal{D})}^{J})\Leftrightarrow (X\cong Y$
in $Sh_{(\mathcal{C},\mathcal{D})}^{K})$

\noindent holds true.
\end{corollary}

An important property of a shape theory is that the shape type of a
\textquotedblleft nice\textquotedblright\ object of $\mathcal{C}$ and its
isomorphism class (in $\mathcal{C}$) coincide. We are to show this property
holds for a $J$-shape theory as well. Let $\mathcal{D}$ be a full and
pro-reflective subcategory of $\mathcal{C}$, let $X\in Ob\mathcal{C}$ and
let $\boldsymbol{p}=(p_{\lambda }):X\rightarrow \boldsymbol{X}=(X_{\lambda
},p_{\lambda \lambda ^{\prime }},\Lambda )$ be a $\mathcal{D}$-expansion of $%
X$. Further, let $J$ be a directed partially ordered set, let $Q\in Ob%
\mathcal{D}$ and let a family $(\varphi ^{j})$ of $\mathcal{C}$-morphisms $%
\varphi ^{j}:X\rightarrow Q$, $j\in J$, be given. We say that $(\varphi
^{j}) $ \emph{uniformly factorizes through} $\boldsymbol{p}$ if there exists
a (fixed) $\lambda \in \Lambda $ such that, for every $j$, $\varphi ^{j}$
factorizes through $p_{\lambda }$. Such a family $(\varphi ^{j})$ determines
a $J$-shape morphism $F:X\rightarrow Q$. Indeed, then there is a $\lambda
\in \Lambda $ such that, for every $j\in J$, there exists a morphism $%
f^{j}:X_{\lambda }\rightarrow Q$ of $\mathcal{D}$ ($\mathcal{D}\subseteq 
\mathcal{C}$ is full) satisfying $\varphi ^{j}=f^{j}p_{\lambda }$. Hence,
the family $(f^{j})$ (with the index function $\{1\}\rightarrow \Lambda $, $%
1\mapsto \lambda $) determines a unique morphism $\boldsymbol{f}=[(f^{j})]:%
\boldsymbol{X}\rightarrow \boldsymbol{Q}=(Q)$ of $pro^{J}$-$\mathcal{D}$.
Since $\boldsymbol{1}:Q\rightarrow \boldsymbol{Q}$ is a $\mathcal{D}$%
-expansion of $Q$, the morphism $\boldsymbol{f}$ determines a unique $J$%
-shape morphism $F=\left\langle \boldsymbol{f}\right\rangle :X\rightarrow Q$
of $Sh_{(\mathcal{C},\mathcal{D})}^{J}$. We say that such an $F$ is \emph{%
induced by} $(\varphi ^{j})$. Notice that the above construction depends on
the index $\lambda $. The converse reads as follows.

\begin{lemma}
\label{L7}Let $X\in Ob\mathcal{C}$, let $\boldsymbol{p}=(p_{\lambda
}):X\rightarrow \boldsymbol{X}=(X_{\lambda },p_{\lambda \lambda ^{\prime
}},\Lambda )$ be a $\mathcal{D}$-expansion of $X$ and let $Q\in Ob\mathcal{D}
$. Then, for every directed partially ordered set $J$, every $J$-shape
morphism $F:X\rightarrow Q$ of $Sh_{(\mathcal{C},\mathcal{D})}^{J}$ is
induced by a family of morphisms $\varphi ^{j}:X\rightarrow Q$ of $\mathcal{C%
}$, $j\in J$, such that $(\varphi ^{j})$ uniformly factorizes through $%
\boldsymbol{p}$.
\end{lemma}

\begin{proof}
Let $F:X\rightarrow Q$ be a $J$-shape morphism of $Sh_{(\mathcal{C},\mathcal{%
D})}^{J}$. For $\mathcal{D}$-expansions $\boldsymbol{p}=(p_{\lambda
}):X\rightarrow \boldsymbol{X}$ and $\boldsymbol{1}:Q\rightarrow \boldsymbol{%
Q}=(Q)$, there exists a representative $\boldsymbol{f}:\boldsymbol{X}%
\rightarrow (Q)$ of $pro^{J}$-$\mathcal{D}$ of $F$. Consequently, there
exist a $\lambda \in \Lambda $ and a family $(f^{j})$ of $\mathcal{D}$%
-morphisms $f^{j}:X_{\lambda }\rightarrow Q$, $j\in J$, which determines $%
\boldsymbol{f}$. Then, by putting $\varphi ^{j}=f^{j}p_{\lambda }$, $j\in J$%
, one obtains the desired inducing family $(\varphi ^{j})$ for $F$.
\end{proof}

Let $(\varphi ^{j})$ and $(\varphi ^{\prime j})$ uniformly factorize through
the same $\mathcal{D}$-expansion $\boldsymbol{p}:X\rightarrow \boldsymbol{X}$
(via a $\lambda $ and a $\lambda ^{\prime }$ respectively). Then $(\varphi
^{j})$ is said \emph{to be almost equal to} $(\varphi ^{\prime j})$, if
there exist a $\lambda _{0}\geq \lambda ,\lambda ^{\prime }$ and a $j_{0}\in
J$ such that

$(\forall j\geq j_{0})$ $\varphi ^{j}p_{\lambda \lambda _{0}}=\varphi
^{\prime j}p_{\lambda ^{\prime }\lambda _{0}}$.

\noindent Clearly, it is an equivalence relation. Further, since $%
\boldsymbol{p}$ is a $\mathcal{D}$-expansion, $(\varphi ^{j})$ and $(\varphi
^{\prime j})$ are almost equal, if and only if there exists a $j_{0}\in J$
such that $\varphi ^{j}=\varphi ^{\prime j}:X\rightarrow Q$, for all $j\geq
j_{0}$.

\begin{lemma}
\label{L8}Let $(\varphi ^{j})$ and $(\varphi ^{\prime j})$ (of $X\in Ob%
\mathcal{C}$ to $Q\in Ob\mathcal{D}$) uniformly factorize through the same $%
\mathcal{D}$-expansion $\boldsymbol{p}:X\rightarrow \boldsymbol{X}$, and let 
$F:X\rightarrow Q$ and $F^{\prime }:X\rightarrow Q$ of $Sh_{(\mathcal{C},%
\mathcal{D})}^{J}$ be induced by $(\varphi ^{j})$ and $(\varphi ^{\prime j})$
respectively. Then $F=F^{\prime }$ if and only if $(\varphi ^{j})$ and $%
(\varphi ^{\prime j})$ are almost equal.
\end{lemma}

\begin{proof}
Let $(\varphi ^{j})$ and $(\varphi ^{\prime j})$ uniformly factorize through
the same $\boldsymbol{p}:X\rightarrow \boldsymbol{X}$, i.e., let there exist 
$\lambda ,\lambda ^{\prime }\in \Lambda $ such that, for every $j\in J$, $%
\varphi ^{j}=f^{j}p_{\lambda }$ and $\varphi ^{\prime j}=f^{\prime
j}p_{\lambda ^{\prime }}$, where $f^{j}:X_{\lambda }\rightarrow Q$ and $%
f^{\prime j}:X_{\lambda ^{\prime }}\rightarrow Q$ are morphisms of $\mathcal{%
D}$. Let $F:X\rightarrow Q$ and $F^{\prime }:X\rightarrow Q$ be the $J$%
-shape morphisms of $Sh_{(\mathcal{C},\mathcal{D})}^{J}$ induced by $%
(\varphi ^{j})$ and $(\varphi ^{\prime j})$ respectively. Let $\boldsymbol{f}%
,\boldsymbol{f}^{\prime }:\boldsymbol{X}\rightarrow \boldsymbol{Q}=(Q)$ of $%
pro^{J}$-$\mathcal{D}$ be representatives of $F,F^{\prime }$ respectively.
Now, if $F=F^{\prime }$ then $\boldsymbol{f}=\boldsymbol{f}^{\prime }$, and $%
\boldsymbol{f}$, $\boldsymbol{f}^{\prime }$ are determined by the families $%
(f^{j})$, $(f^{\prime j})$ respectively. Therefore, there exist a $\lambda
_{0}\geq \lambda ,\lambda ^{\prime }$ and a $j_{0}\in J$ such that

$(\forall j\geq j_{0})$ $f^{j}p_{\lambda \lambda _{0}}=f^{\prime
j}p_{\lambda ^{\prime }\lambda _{0}}$.

\noindent This means that $(\varphi ^{j})$ and $(\varphi ^{\prime j})$ are
almost equal. Conversely, if $(\varphi ^{j})$ and $(\varphi ^{\prime j})$
are almost equal, then the corresponding families $(f^{j})$ and $(f^{\prime
j})$ induce the same morphism $\boldsymbol{f}:\boldsymbol{X}\rightarrow (Q)$
of $pro^{J}$-$\mathcal{D}$. Consequently, the families $(\varphi ^{j})$ and $%
(\varphi ^{\prime j})$ induce the same $J$-shape morphism $F=\left\langle 
\boldsymbol{f}\right\rangle =F^{\prime }:X\rightarrow Q$ of $Sh_{(\mathcal{C}%
,\mathcal{D})}^{J}$.
\end{proof}

Consider now the more special case where $X\equiv P\in Ob\mathcal{D}$ too.
Then $\boldsymbol{1}:P\rightarrow \boldsymbol{P}=(P)$ and $\boldsymbol{1}%
:Q\rightarrow \boldsymbol{Q}=(Q)$ are (the rudimentary) $\mathcal{D}$%
-expansions. Thus, every $J$-shape morphism $F:P\rightarrow Q$ of $Sh_{(%
\mathcal{C},\mathcal{D})}^{J}$ is induced by a family of $\mathcal{D}$%
-morphisms $f^{j}:P\rightarrow Q$, $j\in J$. Furthermore, any two such
families $(f^{j})$, $(f^{\prime j})$ induce the same $J$-ahape morphism, if
and only if $f^{j}=f^{\prime j}$ for almost all $j$ (all $j\geq j_{0}$,
where $j_{0}$ is a fixed index). This implies that there is a surjection

$(\mathcal{D}(P,Q))^{J}\rightarrow Sh_{(\mathcal{C},\mathcal{D})}^{J}(P,Q)$

\noindent of the set of all $J$-families $\Phi =(f^{j})_{j\in J}$ of $%
\mathcal{D}$-morphisms $f^{j}:P\rightarrow Q$ onto the set of all $J$-shape
morphisms $F:P\rightarrow Q$ of $Sh_{(\mathcal{C},\mathcal{D})}^{J}$.
Finally, one can readily see that if an $F:P\rightarrow Q$ is induced by an $%
(f^{j})$ and a $G:Q\rightarrow R$ is induced by a $(g^{j})$, then the
composition $GF:P\rightarrow R$ is induced by $(g^{j}f^{j})$. The following
theorem generalizes Claim 3 of [19].

\begin{theorem}
\label{T5}Let $\mathcal{D}$ be a pro reflective subcategory of $\mathcal{C}$
and let $J$ be a directed partially ordered set. Then, for every pair $%
P,Q\in Ob\mathcal{D}$, the following statements are equivalent:

(i) $P$ and $Q$ are isomorphic objects of $\mathcal{D}$, $P\cong Q$ in $%
\mathcal{D}\subseteq \mathcal{C}$; ;

(ii) $P$ and $Q$ have the same shape, $Sh(P)=Sh(Q)$, i.e., $P\cong Q$ in $%
Sh_{(\mathcal{CD})}$;

(iii) $P$ and $Q$ have the same $J$-shape, $Sh^{J}(P)=Sh^{J}(Q)$, i.e., $%
P\cong Q$ in $Sh_{(\mathcal{CD})}^{J}$
\end{theorem}

\begin{proof}
The equivalence (i) $\Leftrightarrow $ (ii) is the well known fact. The
implication (ii) $\Rightarrow $ (iii) follows by the functor $I_{(\mathcal{C}%
,\mathcal{D})}:Sh_{(\mathcal{C},\mathcal{D})}\rightarrow Sh_{\mathcal{C},%
\mathcal{D})}^{J}$. Let $P,Q\in Ob\mathcal{D}$ have the same $J$-shape. Then
there exists a pair of $J$-shape isomorphisms $F:P\rightarrow Q$, $%
G:Q\rightarrow P$ such that $GF=1_{P}$ and $FG=1_{Q}$ in $Sh_{(\mathcal{C},%
\mathcal{D})}^{J}$. By the above consideration, there exist families $%
(f^{j}) $ and $(g^{j})$ of $\mathcal{D}$-morphisms $f^{j}:P\rightarrow Q$
and $g^{j}:Q\rightarrow P$, $j\in J$, which induce $F$ and $G$ respectively.
Furthermore, the families $(g^{j}f^{j})$ and $(f^{j}g^{j})$ induce $1_{P}$
and $1_{Q}$ (of $Sh_{(\mathcal{C},\mathcal{D})}^{J}$). Since the constant
family $(1_{P}^{j}=1_{P})$ and $(1_{Q}^{j}=1_{Q})$ also induce $1_{P}$ and $%
1_{Q}$ (of $Sh_{(\mathcal{C},\mathcal{D})}^{J}$) respectively, Lemma 8
implies that $g^{j}f^{j}=1_{P}$ and $f^{j}g^{j}=1_{Q}$ hold for almost all $%
j\in J$. Consequently, $P$ and $Q$ are isomorphic objects of $\mathcal{D}$,
and thus, (iii) $\Rightarrow $ (i)$.$
\end{proof}

\section{The continuity theorem for $J$-shape}

A very important benefit of the standard shape theory comparing to the
homotopy theory is the continuity property, i.e., the category $Sh$ admits
the limit functor, while it fails for $HTop$. Moreover, in general, every
(abstract) shape theory has the continuity property (Theorem I.2.6. of
[24]). Further, the continuity property holds for every coarse and every
weak shape theory (Theorems 1 and 2 of [31]). We shall prove hereby that
every $J$-shape theory has the continuity property as well.

\begin{theorem}
\label{T6)}Let $\mathcal{D}$ be a pro-reflective subcategory of $\mathcal{C}$
and let $J$ be a directed partially ordered set. Let $X,Y\in Ob\mathcal{C}$,
let $\boldsymbol{q}=(q_{\mu }):Y\rightarrow \boldsymbol{Y}=(Y_{\mu },q_{\mu
\mu ^{\prime }},M)$ be a $\mathcal{C}$-expansion of $Y$ and let $\boldsymbol{%
H}=(H_{\mu }):X\rightarrow S^{J}(\boldsymbol{Y})$ be a morphism of $pro$-$%
Sh_{(\mathcal{C},\mathcal{D})}^{J}$. Then there exists a unique $J$-shape
morphism $F:X\rightarrow Y$ such that $\boldsymbol{H}=\boldsymbol{Q}F$,
where $\boldsymbol{Q}=(Q_{\mu })=S^{J}(\boldsymbol{q}):Y\rightarrow S^{J}(%
\boldsymbol{Y})$ is the morphism of $pro$-$Sh_{(\mathcal{C},\mathcal{D}%
)}^{J} $ induced by $\boldsymbol{q}$, i.e., for every $\mu \in M$, $H_{\mu
}=Q_{\mu }F$, and $Q_{\mu }$ is induced by $q_{\mu }$, $Q_{\mu
}=S^{J}(q_{\mu })$. In other words, if $\boldsymbol{q}:Y\rightarrow 
\boldsymbol{Y}$ is a $\mathcal{C}$-expansion, then $\boldsymbol{Q}=S^{J}(%
\boldsymbol{q}):Y\rightarrow S^{J}(\boldsymbol{Y})$ is an inverse limit in $%
Sh_{(\mathcal{C},\mathcal{D})}^{J}$, i.e., every $\mathcal{C}$-expansion $%
\boldsymbol{q}:Y\rightarrow \boldsymbol{Y}$ induces, for each $X$, a
bijection

$pro$-$Sh_{(\mathcal{C},\mathcal{D})}^{J}(\left\lfloor X\right\rfloor ,S^{J}(%
\boldsymbol{Y}))\approx Sh_{(\mathcal{C},\mathcal{D})}^{J}(X,Y)$,

\noindent defined by the following diagram \smallskip

$%
\begin{array}{ccc}
&  & X \\ 
& \boldsymbol{H}\swarrow & \downarrow F \\ 
S^{J}(\boldsymbol{Y}) & \underset{S^{J}(\boldsymbol{q})}{\leftarrow } & Y%
\end{array}%
$. \smallskip
\end{theorem}

The proof consists of two steps. In the first one we consider the special
case of a $\mathcal{D}$-expansion $\boldsymbol{q}:Y\rightarrow \boldsymbol{Y}
$.

\begin{lemma}
\label{L9}Let $\mathcal{D}$ be a pro-reflective subcategory of $\mathcal{C}$
and let $J$ be a directed partially ordered set. Let $X,Y\in ObC$, let $%
\boldsymbol{q}=(q_{\mu }):Y\rightarrow \boldsymbol{Y}=(Y_{\mu },q_{\mu \mu
^{\prime }},M)$ be a $\mathcal{D}$-expansion of $Y$ and let $\boldsymbol{H}%
=(H_{\mu }):X\rightarrow S^{J}(\boldsymbol{Y})$ be a morphism of $pro$-$Sh_{(%
\mathcal{C},\mathcal{D})}^{J}$. Then there exists a unique morphism $%
F:X\rightarrow Y$ of $Sh_{(\mathcal{C},\mathcal{D})}^{J}$ such that for
every $\mu \in M$, $H_{\mu }=S^{J}(q_{\mu })F$.
\end{lemma}

\begin{proof}
Let $X,Y\in Ob\mathcal{C}$ and let $\boldsymbol{q}=(q_{\mu }):Y\rightarrow 
\boldsymbol{Y}=(Y_{\mu },q_{\mu \mu ^{\prime }},M)$ be a $\mathcal{D}$%
-expansion of $Y$. Let $\boldsymbol{H}=(H_{\mu }):X\rightarrow S^{J}(%
\boldsymbol{Y})$ be a morphism of $pro$-$Sh_{(\mathcal{C},\mathcal{D})}^{J}$
such that, for every related pair $\mu \leq \mu ^{\prime }$, $H_{\mu
}=S^{J}(q_{\mu \mu ^{\prime }})H_{\mu ^{\prime }}$. Let $\boldsymbol{p}%
=(p_{\lambda }):X\rightarrow \boldsymbol{X}=(X_{\lambda },p_{\lambda \lambda
^{\prime }},\Lambda )$ be a $\mathcal{D}$-expansion of $X$. Since every $%
Y_{\mu }\in Ob\mathcal{D}$, every $J$-shape morphism $H_{\mu }$ is
represented by a unique morphism $\boldsymbol{f}_{\mu }=[(f_{\mu }^{j})]:%
\boldsymbol{X}\rightarrow \left\lfloor Y_{\mu }\right\rfloor $ of $pro^{J}$-$%
\mathcal{D}(\boldsymbol{X},\left\lfloor Y_{\mu }\right\rfloor )$ ($%
\left\lfloor Y_{\mu }\right\rfloor $ is the rudimentary system associated
with $Y_{\mu }$), via the following diagram \smallskip

$%
\begin{array}{ccc}
\boldsymbol{X} & \overset{\boldsymbol{p}}{\leftarrow } & X \\ 
\boldsymbol{f}_{\mu }\downarrow &  & \downarrow H_{\mu } \\ 
\left\lfloor Y_{\mu }\right\rfloor & \underset{\boldsymbol{1}_{Y_{\mu }}}{%
\leftarrow } & Y_{\mu }%
\end{array}%
$, \smallskip

\noindent where $f_{\mu }^{j}:X_{\lambda (\mu )}\rightarrow Y_{\mu }$, $j\in
J$, are morphisms of $\mathcal{D}$. Further, since $H_{\mu }=S^{J}(q_{\mu
\mu ^{\prime }})H_{\mu ^{\prime }}$, $\mu \leq \mu ^{\prime }$, and $%
S^{J}(q_{\mu \mu ^{\prime }})$ is represented by $\boldsymbol{q}_{\mu \mu
^{\prime }}=[(q_{\mu \mu ^{\prime }}^{j}=q_{\mu \mu ^{\prime }})]$, i.e.,
\smallskip

$%
\begin{array}{ccc}
\left\lfloor Y_{\mu }\right\rfloor & \overset{\boldsymbol{1}_{Y_{\mu }}}{%
\leftarrow } & Y_{\mu } \\ 
\boldsymbol{q}_{\mu \mu ^{\prime }}\downarrow &  & \downarrow S^{J}(q_{\mu
\mu ^{\prime }}) \\ 
\left\lfloor Y_{\mu ^{\prime }}\right\rfloor & \underset{\boldsymbol{1}%
_{Y_{\mu ^{\prime }}}}{\leftarrow } & Y_{\mu ^{\prime }}%
\end{array}%
$, \smallskip

\noindent it follows that $\boldsymbol{f}_{\mu }=\boldsymbol{q}_{\mu \mu
^{\prime }}\boldsymbol{f}_{\mu ^{\prime }}$ in $pro^{J}$-$\mathcal{C}$, $\mu
\leq \mu ^{\prime }$. Thus, the following diagram in $pro^{J}$-$\mathcal{D}$
commutes \smallskip

$%
\begin{array}{ccc}
& \boldsymbol{X} &  \\ 
\cdots \text{ }\boldsymbol{f}_{\mu }\swarrow &  & \searrow \boldsymbol{f}%
_{\mu ^{\prime }}\text{ }\cdots \\ 
\cdots \text{ }\leftarrow \left\lfloor Y_{\mu }\right\rfloor & \underset{%
\boldsymbol{q}_{\mu \mu ^{\prime }}}{\leftarrow } & \left\lfloor Y_{\mu
^{\prime }}\right\rfloor \leftarrow \text{ }\cdots%
\end{array}%
$ \smallskip

\noindent This means that, for every pair $\mu \leq \mu ^{\prime }$, there
exist a $\lambda \geq \lambda (\mu ),\lambda (\mu ^{\prime })$ and a $j\in J$
such that, for every $j^{\prime }\geq j$, the following diagram in $\mathcal{%
D}$ commutes: \smallskip

$%
\begin{array}{cccc}
&  &  & X_{\lambda } \\ 
& p_{\lambda (\mu )\lambda }\swarrow &  & \swarrow p_{\lambda (\mu ^{\prime
})\lambda } \\ 
X_{\lambda (\mu )} &  & X_{\lambda (\mu ^{\prime })} &  \\ 
f_{\mu }^{j^{\prime }}\downarrow &  & \downarrow f_{\mu ^{\prime
}}^{j^{\prime }} &  \\ 
Y_{\mu } & \underset{q_{\mu \mu ^{\prime }}}{\leftarrow } & Y_{\mu ^{\prime
}} & 
\end{array}%
$. \smallskip

\noindent Let us define a function $f:M\rightarrow \Lambda $ by putting $%
f(\mu )=\lambda (\mu )$. Then the ordered pair $(f,(f_{\mu }^{j})_{\mu \in
M,j\in J})$ determines a $J$-morphism $(f,f_{\mu }^{j})$ of $\boldsymbol{X}$
to $\boldsymbol{Y}$ of $inv^{J}$-$\mathcal{D}$. Thus, the class $\boldsymbol{%
f}=[(f,f_{\mu }^{j})]:\boldsymbol{X}\rightarrow \boldsymbol{Y}$ is a
morphism of $pro^{J}$-$\mathcal{D}$. Since $\boldsymbol{p}:X\rightarrow 
\boldsymbol{X}$ and $\boldsymbol{q}:Y\rightarrow \boldsymbol{Y}$ are $%
\mathcal{D}$-expansions, the diagram \smallskip

$%
\begin{array}{ccc}
\boldsymbol{X} & \overset{\boldsymbol{p}}{\leftarrow } & X \\ 
\boldsymbol{f}\downarrow &  &  \\ 
\boldsymbol{Y} & \underset{\boldsymbol{q}}{\leftarrow } & Y%
\end{array}%
$ \smallskip

\noindent represents a unique $J$-shape morphism $F:X\rightarrow Y$. Notice
that, by construction,

$S^{J}(q_{\mu })F=H_{\mu }$

\noindent holds for every $\mu \in M$. Moreover, such an $F$ is unique
because $\boldsymbol{q}:Y\rightarrow \boldsymbol{Y}$ is a $\mathcal{D}$%
-expansion. Therefore, for every $X$, the correspondence $\boldsymbol{H}%
=(H_{\mu })\mapsto F$, induced by $\boldsymbol{q}$, defines a bijection of $%
pro$-$Sh_{(\mathcal{C},\mathcal{D})}^{J}(\left\lfloor X\right\rfloor ,%
\boldsymbol{Y})$ onto $Sh_{(\mathcal{C},\mathcal{D})}^{J}(X,Y)$.
\end{proof}

\begin{proof}
(of Theorem 6) Let $\boldsymbol{q}=(q_{\mu }):Y\rightarrow \boldsymbol{Y}%
=(Y_{\mu },q_{\mu \mu ^{\prime }},M)$ be a $\mathcal{C}$-expansion of $Y$.
Firstly, if $F:X\rightarrow Y$ is a morphism of $Sh_{(\mathcal{C},\mathcal{D}%
)}^{J}$, then all $F_{\mu }=S^{J}(q_{\mu })F$, $\mu \in M$, define a
morphism $\boldsymbol{H}=(F_{\mu }):\left\lfloor X\right\rfloor \rightarrow
S^{J}(\boldsymbol{Y})$ of $pro$-$Sh_{(\mathcal{C},\mathcal{D})}^{J}$,
because $F_{\mu }=S^{J}(q_{\mu \mu ^{\prime }})F_{\mu ^{\prime }}$, $\mu
\leq \mu ^{\prime }$. Conversely, let an $\boldsymbol{H}=(H_{\mu })\in pro$-$%
Sh_{(\mathcal{C},\mathcal{D})}^{J}(\left\lfloor X\right\rfloor ,S^{J}(%
\boldsymbol{Y}))$ be given. Choose any $\mathcal{D}$-expansion

$\boldsymbol{q}^{\prime }=(q_{\nu }^{\prime }):Y\rightarrow \boldsymbol{Y}%
^{\prime }=(Y_{\nu }^{\prime },q_{\nu \nu ^{\prime }}^{\prime },N)$

\noindent of $Y$ ($\mathcal{D}$ is a pro-reflective subcategory of $\mathcal{%
C}$!). Since $\boldsymbol{q}$ is a $\mathcal{C}$-expansion (with respect to $%
\mathcal{D}$), there exists a unique $\boldsymbol{g}:\boldsymbol{Y}%
\rightarrow \boldsymbol{Y}^{\prime }$ of $pro$-$\mathcal{C}$ such that $%
\boldsymbol{gq}=\boldsymbol{q}^{\prime }$. Let $(g,g_{\nu })$ be a
representative of $\boldsymbol{g}$ in $inv$-$\mathcal{C}$. For every $\nu
\in N$, denote by $G_{\nu }:Y_{g(\nu )}\rightarrow Y_{\nu }^{\prime }$ the
morphism of $Sh_{(\mathcal{C},\mathcal{D})}^{J}$ induced by $g_{\nu }$,
i.e., $G_{\nu }=S^{J}(g_{\nu })$. Similarly, denote $Q_{\nu }^{\prime
}=S^{J}(q_{\nu }^{\prime }):Y\rightarrow Y_{\nu }^{\prime }$ and $Q_{\nu \nu
^{\prime }}^{\prime }=S^{J}(q_{\nu \nu ^{\prime }}^{\prime }):Y_{\nu
^{\prime }}^{\prime }\rightarrow Y_{\nu }^{\prime }$, $\nu \leq \nu ^{\prime
}$. Then, since $(g,g_{\nu })$ is a morphism of $inv$-$\mathcal{C}$, one
readily sees that $(g,G_{\nu }):S^{J}(\boldsymbol{Y})\rightarrow S^{J}(%
\boldsymbol{Y}^{\prime })$ is a morphism of $inv$-$Sh_{(\mathcal{C},\mathcal{%
D})}^{J}$. Thus, the equivalence class $\boldsymbol{G}=[(g,G_{\nu })]:S^{J}(%
\boldsymbol{Y})\rightarrow S^{J}(\boldsymbol{Y}^{\prime })$ is a morphism of 
$pro$-$Sh_{(\mathcal{C},\mathcal{D})}^{J}$. Let $\boldsymbol{F}=(F_{\nu
}):\left\lfloor X\right\rfloor \rightarrow S^{J}(\boldsymbol{Y}^{\prime })$
of $pro$-$Sh_{(\mathcal{C},\mathcal{D})}^{J}$ be the composition of $%
\boldsymbol{H}$ and $\boldsymbol{G}$. Then $F_{\nu }=G_{\nu }H_{g(\nu )}$, $%
\nu \in N$, and $F_{\nu }=Q_{\nu \nu ^{\prime }}^{\prime }F_{\nu ^{\prime }}$%
, $\nu \leq \nu ^{\prime }$. By Lemma 9, there exists a unique $%
F:X\rightarrow Y$ of $Sh_{(\mathcal{C},\mathcal{D})}^{J}$ such that, for
every $\nu \in N$, $Q_{\nu }^{\prime }F=F_{\nu }$. This means that, for
every $\nu \in N$,

$Q_{\nu }^{\prime }F=G_{\nu }H_{g(\nu )}$, i.e.,

$S^{J}(q_{\nu }^{\prime })F=S^{J}(g_{\nu })H_{g(\nu )}$.

\noindent We have to prove that, for every $\mu \in M$, $S^{J}(q_{\mu
})F=H_{\mu }$ holds. Firstly, we will prove the following statement:

$(\forall \mu \in M)(\forall P\in Ob\mathcal{D})(\forall u\in \mathcal{C}%
(Y_{\mu },P)$

$S^{J}(u)S^{J}(q_{\mu })F=S^{J}(u)H_{\mu }$.

\noindent Notice that a $u:Y_{\mu }\rightarrow P$ of $\mathcal{C}$ yields a
unique $\boldsymbol{u}=[(u)]:\boldsymbol{Y}\rightarrow \left\lfloor
P\right\rfloor $ of $pro$-$\mathcal{C}$. Observe that $\boldsymbol{uq}$ is a
(rudimentary) morphism $uq_{\mu }:Y\rightarrow P$ belonging to $\mathcal{C}$%
. Since $\boldsymbol{q}^{\prime }$ is a $\mathcal{D}$-expansion and $P\in Ob%
\mathcal{D}$, there exists a unique $\boldsymbol{v}:\boldsymbol{Y}^{\prime
}\rightarrow \left\lfloor P\right\rfloor $ of $pro$-$\mathcal{D}$
(represented by a $v_{\nu }:Y_{\nu }^{\prime }\rightarrow P$ of $\mathcal{D}$%
) such that $\boldsymbol{vq}^{\prime }=\boldsymbol{uq}$. Then, $\boldsymbol{%
uq}=\boldsymbol{vgq}$, which implies that $\boldsymbol{u}=\boldsymbol{vg}$.
This means that there exists a $\mu ^{\prime }\geq \mu ,g(\nu )$ such that

$uq_{\mu \mu ^{\prime }}=vg_{\nu }q_{g(\nu )\mu ^{\prime }}$.

\noindent Now one calculates in a straightforward way that

$S^{J}(u)S^{J}(q_{\mu })F=$

\noindent $S^{J}(u)S^{J}(q_{\mu \mu ^{\prime }})S^{J}(q_{\mu ^{\prime
}})F=S^{J}(uq_{\mu \mu ^{\prime }})S^{J}(q_{\mu ^{\prime }})F=S^{J}(vg_{\nu
}q_{g(\nu )\mu ^{\prime }})S^{J}(q_{\mu ^{\prime }})F$

\noindent $=S^{J}(v)S^{J}(g_{\nu }q_{g(\nu )\mu ^{\prime }}q_{\mu ^{\prime
}})F=S^{J}(v)S^{J}(g_{\nu }q_{g(\nu )})F=S^{J}(v)S^{J}(q_{\nu }^{\prime })F$

\noindent $=S^{J}(v)S^{J}(g_{\nu })H_{g(\nu )}=S^{J}(v)S^{J}(g_{\nu
})S^{J}(q_{g(\nu )\mu ^{\prime }})H_{\mu ^{\prime }}=S^{J}(vg_{\nu }q_{g(\nu
)\mu ^{\prime }})H_{\mu ^{\prime }}$

\noindent $=S^{J}(uq_{\mu \mu ^{\prime }})H_{\mu ^{\prime
}}=S^{J}(u)S^{J}(q_{\mu \mu ^{\prime }})H_{\mu ^{\prime }}=S^{J}(u)H_{\mu }$,

\noindent which proves the statement. Given a $\mu \in M$, let

$\boldsymbol{q}^{\mu }=(q_{\alpha }^{\mu }):Y_{\mu }\rightarrow \boldsymbol{Y%
}^{\mu }=(Y_{\alpha }^{\mu },q_{\alpha \alpha ^{\prime }}^{\mu },A^{\mu })$

\noindent be a $\mathcal{D}$-expansion of $Y_{\mu }$. Then, by the above
statement, for every $\alpha \in A^{\mu }$,

$S^{J}(q_{\alpha }^{\mu })S^{J}(q_{\mu })F=S^{J}(q_{\alpha }^{\mu })H_{\mu }$%
.

\noindent According to the definition of the coarse shape category $Sh_{(%
\mathcal{C},\mathcal{D})}^{J}$, this means that the coarse shape morphisms

$S^{J}(q_{\mu })F,H_{\mu }:X\rightarrow Y_{\mu }$

\noindent admit the same representing morphism $\boldsymbol{f}:\boldsymbol{X}%
\rightarrow \boldsymbol{Y}^{\mu }$ of $pro^{J}$-$\mathcal{D}$. Thus,

$S^{J}(q_{\mu })F=H_{\mu }$.

\noindent Finally, such an $F$ is unique because

$S^{J}(q_{\mu })F=S^{J}(q_{\mu })F^{\prime }$, $\mu \in M$,

\noindent immediately implies

$S^{J}(q_{\nu }^{\prime })F=S^{J}(q_{\nu }^{\prime })F^{\prime }$, $\nu \in
N $,

\noindent which means that $F=F^{\prime }$.
\end{proof}

\section{A $J$-shape isomorphism}

In this section, we are going to establish an analogue of the well known
Morita lemma of [26], which should characterize a $J$-shape isomorphism in
an elegant and rather operative manner. According to the \textquotedblleft
reindexing theorem\textquotedblright\ (Theorem 2.) and definition of the
abstract $J$-shape category $Sh_{(\mathcal{C},\mathcal{D})}^{J}$, it
suffices to characterize an isomorphism $\boldsymbol{f}\in pro^{J}$-$%
\mathcal{D}(\boldsymbol{X},\boldsymbol{Y})$ which admits a level
representative $(1_{\Lambda },f_{\lambda }^{j}):\boldsymbol{X}\rightarrow 
\boldsymbol{Y}$ of $inv^{J}$-$\mathcal{D}$. In the case of inverse
sequences, a strictly increasing simple representative will do. Since the
characterization does not depend on the objects of $\mathcal{D}$, we shall
consider such an $\boldsymbol{f}$ of $pro^{J}$-$\mathcal{C}$ as well as the
special case of $tow^{\mathbb{N}}$-$\mathcal{C}$.

\begin{theorem}
\label{T7}Let $\mathcal{C}$ be a category and let $J$ be a directed
partially ordered set. Let $\boldsymbol{\boldsymbol{X}}=(X_{\lambda
},p_{\lambda \lambda ^{\prime }},\Lambda )$ and $\boldsymbol{Y}=(Y_{\lambda
},q_{\lambda \lambda ^{\prime }},\Lambda )$ be inverse systems in $\mathcal{C%
}$ over the same index set $\Lambda $ and let a morphism $\boldsymbol{f}:%
\boldsymbol{X}\rightarrow \boldsymbol{Y}$ of $pro^{J}$-$\mathcal{C}$ admit a
level representative $(1_{\Lambda },f_{\lambda }^{j})$. Then $\boldsymbol{f}$
is an isomorphism if and only if, for every $\lambda \in \Lambda $, there
exist a $\lambda ^{\prime }\geq \lambda $ and a $j_{\lambda }\in J$ such
that, for every $j\geq j_{\lambda }$, there exists a $\mathcal{C}$-morphism $%
h_{\lambda }^{j}:Y_{\lambda ^{\prime }}\rightarrow X_{\lambda }$ so that the
following diagram in $\mathcal{C}$ commutes:

$%
\begin{array}{ccc}
X_{\lambda } & \longleftarrow & X_{\lambda ^{\prime }} \\ 
f_{\lambda }^{j}\downarrow & h_{\lambda }^{j}\nwarrow & \downarrow
f_{\lambda ^{\prime }}^{j} \\ 
Y_{\lambda } & \longleftarrow & Y_{\lambda ^{\prime }}%
\end{array}%
$.
\end{theorem}

\begin{proof}
Let $\boldsymbol{f}:\boldsymbol{X}\rightarrow \boldsymbol{Y}$ be an
isomorphism of $pro^{J}$-$\mathcal{C}$ which admits a level representative $%
(1_{\Lambda },f_{\lambda }^{j})$. Let $\boldsymbol{f}^{-1}\equiv \boldsymbol{%
g}=[(g,g_{\lambda }^{j})]:\boldsymbol{Y}\rightarrow \boldsymbol{X}$ be the
inverse of $\boldsymbol{f}$, i.e.

$(g,g_{\lambda }^{j})(1_{\Lambda },f_{\lambda }^{j})\sim (1_{\Lambda
},1_{X_{\lambda }})$ $\wedge $ $(1_{\Lambda },f_{\lambda }^{j})(g,g_{\lambda
}^{j})\sim (1_{\Lambda },1_{Y_{\lambda }})$.

\noindent Given any $\lambda \in \Lambda $, choose $\lambda _{1}^{\prime
},\lambda _{2}^{\prime }\in \Lambda $ according to the above equivalence
relations. Then there exists a $\lambda ^{\prime }\geq \lambda _{1}^{\prime
},\lambda _{2}^{\prime }$. Thus $\lambda ^{\prime }\geq \lambda ,g(\lambda )$%
. Further, choose $j_{1},j_{2}\in \mathbb{N}$ according to the above
equivalence relations and the given $\lambda $. Since $(1_{\Lambda
},f_{\lambda }^{j})$ is an $J$-morphism, for the pair $g(\lambda )\leq
\lambda ^{\prime }$, there exists a $j_{3}\in J$ such that the appropriate
commutativity condition holds. Since $J$ is directed, there exists a $%
j_{\lambda }\geq j_{1},j_{2},j_{3}$. Let us define, for every $j\geq
j_{\lambda }$, a morphism $h_{\lambda }^{j}:Y_{\lambda ^{\prime
}}\rightarrow X_{\lambda }$ of $\mathcal{C}$ by putting

$h_{\lambda }^{j}=g_{\lambda }^{j}q_{g(\lambda )\lambda ^{\prime }}$.

\noindent We are to prove that the needed diagram ommutes. Firstly,
according to the second equivalence relation,

$f_{\lambda }^{j}h_{\lambda }^{j}=f_{\lambda }^{j}g_{\lambda
}^{j}q_{g(\lambda )\lambda ^{\prime }}=q_{\lambda \lambda ^{\prime }}$.

\noindent Thus, the left (lower) triangle in the diagram commutes. Further,
since $j\geqslant j_{3}$,

$h_{\lambda }^{j}f_{\lambda ^{\prime }}^{j}=g_{\lambda }^{j}q_{g(\lambda
)\lambda ^{\prime }}f_{\lambda ^{\prime }}^{j}=g_{\lambda }^{j}f_{g(\lambda
)}^{j}p_{g(\lambda )\lambda ^{\prime }}$,

\noindent while, according to the first equivalence relation,

$g_{\lambda }^{j}f_{g(\lambda )}^{j}p_{g(\lambda )\lambda ^{\prime
}}=p_{\lambda \lambda ^{\prime }}$.

\noindent Therefore,

$h_{\lambda }^{j}f_{\lambda ^{\prime }}^{j}=p_{\lambda \lambda ^{\prime }}$,

\noindent which proves commutativity of the right (upper) triangle in the
diagram.

\noindent Conversely, suppose that a morphism $\boldsymbol{f}=[(1_{\Lambda
},f_{\lambda }^{j})]:\boldsymbol{X}\rightarrow \boldsymbol{Y}$ of $pro^{J}$-$%
\mathcal{C}$ fulfils the condition of the theorem. Let $g:\Lambda
\rightarrow \Lambda $ be defined by that condition, i.e., for each $\lambda $%
, choose and fix a $g(\lambda )=\lambda ^{\prime }\geq \lambda $ by the
condition. Further, for each $\lambda \in \Lambda $, choose amd fix a $%
j_{\lambda }\in J$ by the same condition. Let us define, for each $\lambda
\in \Lambda $ and every $j\in J$, a morphism $g_{\lambda }^{j}:Y_{g(\lambda
)}\rightarrow X_{\lambda }$ of $\mathcal{C}$ by putting

$g_{\lambda }^{j}=\left\{ 
\begin{array}{c}
h_{\lambda }^{j_{\lambda }};\text{ }j\ngeq j_{\lambda } \\ 
h_{\lambda }^{j};\text{ }j\geqslant j_{\lambda }%
\end{array}%
\right. $,

\noindent where $h_{\lambda }^{j}$ comes from the condition. We have to
prove that $\left( g,g_{\lambda }^{j}\right) :\boldsymbol{Y}\rightarrow 
\boldsymbol{X}$ is a $J$-morphism. Let a pair $\lambda \leq \lambda ^{\prime
}$ be given. Choose a $\lambda _{0}\geq g(\lambda ),g(\lambda ^{\prime })$
and put $\lambda _{1}=g(\lambda _{0})$. Since $(1_{\Lambda },f_{\lambda
}^{j})$ is a $J$-morphism, for the pairs $g\left( \lambda \right) \leq
\lambda _{0}$ and $g\left( \lambda ^{\prime }\right) \leq \lambda _{0}$,
there exist $j_{1},j_{2}\in J$ such that the appropriate commutativity
conditions hold respectively. Since $J$ is directed, there exists a

$j\geq j_{\lambda },j_{\lambda ^{\prime }},j_{\lambda _{0}},j_{1},j_{2}$.

\noindent Now, for every $j^{\prime }\geq j$, consider the following
corresponding diagram:%
\begin{equation}
\begin{tabular}{lllllllll}
$X_{\lambda }$ & $\longleftarrow $ & $X_{\lambda ^{\prime }}$ & $%
\longleftarrow $ & $X_{g(\lambda ^{\prime })}$ &  &  &  &  \\ 
& $\nwarrow $ &  & $\nwarrow $ & $\downarrow $ & $\nwarrow $ &  &  &  \\ 
&  & $X_{g(\lambda )}$ &  & $\longleftarrow $ &  & $X_{\lambda _{0}}$ &  & 
\\ 
&  & $\downarrow $ &  & $Y_{g(\lambda ^{\prime })}$ &  & $\downarrow $ &  & 
\\ 
& $\nwarrow $ &  &  &  & $\nwarrow $ &  & $\nwarrow $ &  \\ 
&  & $Y_{g(\lambda )}$ &  & $\longleftarrow $ &  & $Y_{\lambda _{0}}$ & $%
\longleftarrow $ & $Y_{\lambda _{1}}$%
\end{tabular}%
\text{. }  \tag{1}
\end{equation}%
\noindent We shall prove, by chasing diagram $(1)$, that 
\begin{equation}
g_{\lambda }^{j^{\prime }}q_{g(\lambda )\lambda _{1}}=p_{\lambda \lambda
^{\prime }}g_{\lambda ^{\prime }}^{j^{\prime }}q_{g(\lambda ^{\prime
})\lambda _{1}}\text{.}  \tag{2}
\end{equation}%
\noindent Since $j^{\prime }\geqslant j_{\lambda _{0}},$ the condition of
the theorem implies 
\begin{equation}
g_{\lambda }^{j^{\prime }}q_{g(\lambda )\lambda _{1}}=h_{\lambda
}^{j^{\prime }}q_{g(\lambda )\lambda _{0}}f_{\lambda _{0}}^{j^{\prime
}}h_{\lambda _{0}}^{j^{\prime }}\text{. }  \tag{3}
\end{equation}%
\noindent Since $j^{\prime }\geqslant j_{1}$, 
\begin{equation}
h_{\lambda }^{j^{\prime }}q_{g(\lambda )\lambda _{0}}f_{\lambda
_{0}}^{j^{\prime }}h_{\lambda _{0}}^{j^{\prime }}=h_{\lambda }^{j^{\prime
}}f_{g(\lambda )}^{j^{\prime }}p_{g(\lambda )\lambda _{0}}h_{\lambda
_{0}}^{j^{\prime }}\text{. }  \tag{4}
\end{equation}%
\noindent Since $j^{\prime }\geqslant j_{\lambda },j_{\lambda ^{\prime }},$
the condition of the theorem implies 
\begin{equation}
h_{\lambda }^{j^{\prime }}f_{g(\lambda )}^{j^{\prime }}p_{g(\lambda )\lambda
_{0}}h_{\lambda _{0}}^{j^{\prime }}=p_{\lambda \lambda _{0}}h_{\lambda
_{0}}^{j^{\prime }}=p_{\lambda \lambda ^{\prime }}h_{\lambda ^{\prime
}}^{j^{\prime }}f_{g(\lambda ^{\prime })}^{j^{\prime }}p_{g(\lambda ^{\prime
})\lambda _{0}}h_{\lambda _{0}}^{j^{\prime }}.  \tag{5}
\end{equation}%
\noindent Since $j^{\prime }\geqslant j_{2}$, 
\begin{equation}
p_{\lambda \lambda ^{\prime }}h_{\lambda ^{\prime }}^{j^{\prime
}}f_{g(\lambda ^{\prime })}^{j^{\prime }}p_{g(\lambda ^{\prime })\lambda
_{0}}h_{\lambda _{0}}^{j^{\prime }}=p_{\lambda \lambda ^{\prime }}h_{\lambda
^{\prime }}^{j^{\prime }}q_{g(\lambda ^{\prime })\lambda _{0}}f_{\lambda
_{0}}^{j^{\prime }}h_{\lambda _{0}}^{j^{\prime }}\text{. }  \tag{6}
\end{equation}%
\noindent Finally, since $j^{\prime }\geqslant j_{\lambda _{0}}$, the
condition of the theorem implies 
\begin{equation}
p_{\lambda \lambda ^{\prime }}h_{\lambda ^{\prime }}^{j^{\prime
}}q_{g(\lambda ^{\prime })\lambda _{0}}f_{\lambda _{0}}^{j^{\prime
}}h_{\lambda _{0}}^{j^{\prime }}=p_{\lambda \lambda ^{\prime }}h_{\lambda
^{\prime }}^{j^{\prime }}q_{g(\lambda ^{\prime })g(\lambda _{0})}=p_{\lambda
\lambda ^{\prime }}g_{\lambda ^{\prime }}^{j^{\prime }}q_{g(\lambda ^{\prime
})\lambda _{1}}\text{. }  \tag{11}
\end{equation}%
\noindent Now, by combining $\left( 3\right) ,$ $\left( 4\right) ,$ $\left(
5\right) $, $\left( 6\right) $ and $\left( 7\right) $, one establishes $%
\left( 2\right) $, which proves that $\left( g,g_{\lambda }^{j}\right) $ is
a $J$-morphism. Moreover, by the condition of the theorem, it is readily
seen that, for each $\lambda \in \Lambda $ and every $j^{\prime }\in J$, $%
j^{\prime }\geqslant j_{\lambda }$, 
\begin{equation*}
\text{ }g_{\lambda }^{j^{\prime }}f_{g(\lambda )}^{j^{\prime }}=h_{\lambda
}^{j^{\prime }}f_{g(\lambda )}^{j^{\prime }}=p_{\lambda g(\lambda )}\text{ }%
\wedge \text{ }f_{\lambda }^{j^{\prime }}g_{\lambda }^{j^{\prime
}}=f_{\lambda }^{j^{\prime }}h_{\lambda }^{j^{\prime }}=q_{\lambda g(\lambda
)}\text{. }
\end{equation*}%
\noindent This shows that 
\begin{equation*}
(g,g_{\lambda }^{j})(1_{\Lambda },f_{\lambda }^{j})\sim (1_{\Lambda
},1_{X_{\lambda }})\text{ }\wedge \text{ }(1_{\Lambda },f_{\lambda
}^{j})(g,g_{\lambda }^{j})\sim (1_{\Lambda },1_{Y_{\lambda }})\text{, }
\end{equation*}%
\noindent which means that $\boldsymbol{g}=[(g,g_{\lambda }^{j})]:%
\boldsymbol{Y}\rightarrow \boldsymbol{X}$ is the inverse of $\boldsymbol{f}$%
. Therefore, $\boldsymbol{f}$ is an isomorphism of $pro^{J}$-$\mathcal{C}$.
\end{proof}

\begin{remark}
\label{R3}Since $pro$-$\mathcal{C}=pro^{\{1\}}$-$\mathcal{C}$, the original
Morita lemma is the simpleast case of Theorem 7. Further, since the coarse
shape category $Sh_{(\mathcal{C},\mathcal{D})}^{\ast }$ is the $\mathbb{N}$%
-shape category Sh$_{(\mathcal{C},\mathcal{D})}^{\mathbb{N}}$, Theorem 7 is
a generalization of [19], Theorem 6.1.
\end{remark}

One can easily verify that the condition (of Theorem 7) characterizing an
isomorphism may be reduced to a cofinal subset $\Lambda ^{\prime }\subseteq
\Lambda $. Thus, the following corollary holds.

\begin{corollary}
\label{C4}If an $\boldsymbol{f}=[(1_{\Lambda },f_{\lambda }^{j})]:%
\boldsymbol{X}\rightarrow \boldsymbol{Y}$ of $pro^{J}$-$\mathcal{C}$ admits
a cofinal subset $\Lambda ^{\prime }\subseteq \Lambda $ such that, for every 
$\lambda ^{\prime }\in \Lambda ^{\prime }$, there exists a $j\in J$, so
that, for every $j^{\prime }\geq j$, $f_{\lambda ^{\prime }}^{j^{\prime }}$
is an isomorphism of $\mathcal{C}$, then $\boldsymbol{f}$ is an isomorphism.
\end{corollary}

For the sake of completeness and unifying notations, we include hereby
Theorem 6.4 of [19] (see also [25], Section 2) concerning the special case
of inverse sequences and $J=\mathbb{N}$. It is very useful, for instance, in
detecting an $\mathbb{N}$-shape (i.e., a coarse shape) isomorphism of
metrizable compacta (i.e., in the case $(\mathcal{C},\mathcal{D}%
)=(HcM,HcPol) $).

\begin{theorem}
\label{T8}Let $\boldsymbol{\boldsymbol{X}}=(X_{n},p_{nn^{\prime }})$ and $%
\boldsymbol{Y}=(Y_{m},q_{mm^{\prime \prime }})$ be inverse sequences in a
category $\mathcal{C}$, let $\boldsymbol{f}:\boldsymbol{X}\rightarrow 
\boldsymbol{Y}$ be a morphism of $tow^{\mathbb{N}}$-$\mathcal{C}$ and let $%
(f,f_{m}^{j})$ be any simple representative of $\boldsymbol{f}$ with a
commutativity radius $\gamma $ and $f$ strictly increasing. If for every $%
j\in \mathbb{N}$ and every $m=1,\ldots ,\gamma (j)-1$, there exists a $%
\mathcal{C}$-morphism $h_{f(m)}^{j}:Y_{m+1}\rightarrow X_{f(m)}$ such that
the diagram

$%
\begin{array}{ccc}
X_{f(m)} & \longleftarrow & X_{f(m+1)} \\ 
f_{m}^{j}\downarrow & h_{f(m)}^{j}\nwarrow & \downarrow f_{m+1}^{j} \\ 
Y_{m} & \longleftarrow & Y_{m+1}%
\end{array}%
$

\noindent in $\mathcal{C}$ commutes, then $\boldsymbol{f}$ is an isomorphism
of $tow^{\mathbb{N}}$-$\mathcal{C}$.

\noindent Conversely, if $\boldsymbol{f}$ is an isomorphism of $tow^{\mathbb{%
N}}$-$\mathcal{C}$, then, for every $m\in \mathbb{N}$, there exist an $%
m^{\prime }\geq m$ and a $j\in \mathbb{N}$ such that, for every $j^{\prime
}\geq j$, there exists a $\mathcal{C}$-morphism $h_{f(m)}^{j^{\prime
}}:Y_{m^{\prime }}\rightarrow X_{f(m)}$ so that the following diagram

$%
\begin{array}{ccc}
X_{f(m)} & \longleftarrow & X_{f(m^{\prime })} \\ 
f_{m}^{j^{\prime }}\downarrow & h_{f(m)}^{j^{\prime }}\nwarrow & \downarrow
f_{m^{\prime }}^{j^{\prime }} \\ 
Y_{m} & \longleftarrow & Y_{m^{\prime }}%
\end{array}%
$

\noindent in $\mathcal{C}$ commutes.
\end{theorem}

\begin{remark}
\label{R4}We give no additional example but those of [13], [17] and [19],
though one can, by means of them, easily construct some with $J=(\mathbb{N}%
,\leq ^{\prime })$ (for instance, in the case of $m\leq ^{\prime }n$ iff $%
\frac{n}{m}\in \mathbb{N}$) . Nevertheless, an example in the case of an
unbounded infinite $J\neq \mathbb{N}$ would be interesting. \bigskip
\end{remark}

\begin{center}
\textbf{References}
\end{center}

\noindent \lbrack 1] \textit{K. Borsuk}, Concerning homotopy properties of
compacta,\textit{\ }Fund. Math. \textbf{62} (1968) 223-254.

\noindent \lbrack 2] \textit{K. Borsuk,} Theory of Shape, Lecture Notes
Series \textbf{28}, Matematisk Inst. Aarhus, 1971.

\noindent \lbrack 3] \textit{K. Borsuk,} Theory of Shape, Monografie
Matematyczne \textbf{59}, Polish Scientific Publishers, Warszawa, 1975.

\noindent \lbrack 4] \textit{K. Borsuk,} Some quantitative properties of
shapes, Fund. Math. \textbf{93} (1976) 197-212.

\noindent \lbrack 5] \textit{F. W. Cathey,} Strong Shape Theory, Thesis,
University of Washington, 1979.

\noindent \lbrack 6] \textit{T. A. Chapman,} On some applications of
infinite-dimensional manifolds in the theory of shape, Fund. Math. \textbf{76%
} (1072), 261-276.

\noindent \lbrack 7] \textit{T. A. Chapman,} Shapes of finite-dimensional
compacta, Fund. Math. \textbf{76} (1072), 181-193.

\noindent \lbrack 8] \textit{J.-M. Cordier and T. Porter,} Shape Theory:
Categorical Methods of Approximation, Ellis Horwood Ltd., Chichester, 1989.
(Dover edition, 2008.)

\noindent \lbrack 9] \textit{B. \v{C}ervar and N. Ugle\v{s}i\'{c}}, Category
descriptions of the $S_{n}$- and $S$.equivalence, Math. Comm. \textbf{13}
(2008), 1-19.

\noindent \lbrack 10] \textit{J. Dydak and J. Segal,} Shape theory: An
introduction, Lecture Notes in Math. \textbf{688}, Springer-Verlag, Berlin,
1978.

\noindent \lbrack 11] \textit{J. Dydak and J. Segal,} Strong shape theory:
Disertationes Math. \textbf{192} (1981).

\noindent \lbrack 12] \textit{D.A. Edwards and H.M. Hastings,} \v{C}ech and
Steenrod homotopy theories with applications to geometric topology, Lecture
Notes in Math. \textbf{542}, Springer-Verlag, Berlin, 1976.

\noindent \lbrack 13] \textit{K. R. Goodearl and T. B. Rushing,} Direct
limit groups and the Keesling-Marde\v{s}i\'{c} shape fibration, Pacific J.
Math.\textbf{86} (1980), 471-476.

\noindent \lbrack 14] \textit{H. Herlich }and\textit{\ G. E. Strecker,}
Category Theory, Allyn and Bacon Inc., Boston, 1973.

\noindent \lbrack 15] \textit{A. Kadlof, N. Kocei\'{c} Bilan and N. Ugle\v{s}%
i\'{c}}, Borsuk's quasi-equivalence is not transitive, Fund. Math. \textbf{%
197} (2007), 215-227.

\noindent \lbrack 16] \textit{J. Keesling,} Products in the shape category
and some applications, Sym. Math. Instituto Nazionale di Alta Matematica 
\textbf{16} (Roma, 1973), Academic Press, New York, 1974, 133-142.

\noindent \lbrack 17] \textit{J. Keesling and S. Marde\v{s}i\'{c}}, A shape
fibration with fibers of different shape, Pacific J. Math. \textbf{84}
(1979) 319-331.

\noindent \lbrack 18] \textit{N. Kocei\'{c} Bilan,} On some coarse shape
invariants, Top. Appl. \textbf{157} (2010), 2679-2685.

\noindent \lbrack 19] \textit{N. Kocei\'{c} Bilan and N. Ugle\v{s}i\'{c},}
The coarse shape, Glas. Mat. \textbf{42}(\textbf{62}) (2007), 145-187.

\noindent \lbrack 20] \textit{S. Marde\v{s}i\'{c}}, Shapes for topological
spaces, General Topology Appl. \textbf{3} (1973) 265-282.

\noindent \lbrack 21] \textit{S. Marde\v{s}i\'{c}}, Comparing fibres in a
shape fibration, Glasnik Mat. \textbf{13}(\textbf{33}) (1978) 317-333.

\noindent \lbrack 22] \textit{S. Marde\v{s}i\'{c}}, Inverse limits and
resolutions, Shape Theory and Geom. Top. Proc. (Dubrovnik, 1981), Lecture
Notes in Math. \textbf{870}, Springer-Verlag, Berlin, 1981, 239-252.

\noindent \lbrack 23] \textit{S. Marde\v{s}i\'{c} }and\textit{\ J. Segal},%
\textit{\ }Shapes of compacta and ANR--systems, Fund. Math. \textbf{72}
(1971) 41-59.

\noindent \lbrack 24] \textit{S. Marde\v{s}i\'{c} }and\textit{\ J. Segal},
Shape Theory, North Holland, Amsterdam, 1982.

\noindent \lbrack 25] \textit{S. Marde\v{s}i\'{c} and N. Ugle\v{s}i\'{c}}, A
category whose isomorphisms induce an equivalence relation coarser than
shape, Top. Appl. \textbf{153} (2005) 448-463\textit{. }

\noindent \lbrack 26] \textit{K. Morita}, The Hurewicz and the Whitehead
theorems in shape theory, Sci. Reports Tokyo Kyoiku Daigaku, Sec. A, \textbf{%
12} (1974) 246-258.

\noindent \lbrack 27] \textit{K. Morita}, On shapes of topological spaces,
Fund. Math. \textbf{86} (1975) 251-259.

\noindent \lbrack 28] \textit{D. G. Quillen,} Homotopical Algebra, Lecture
Notes in Math. \textbf{43}, Springer-Verlag, Berlin, 1967.

\noindent \lbrack 29] \textit{N. Ugle\v{s}i\'{c},} Stability is a weqk shape
invariant, Glasnik Math. \textbf{44}(\textbf{64}) (2009), 241-254.

\noindent \lbrack 30] \textit{N. Ugle\v{s}i\'{c},} Classifications coarser
than shape, Math. Comm. \textbf{13} (2008), 193-213.

\noindent \lbrack 31] \textit{N. Ugle\v{s}i\'{c},} Continuity in the Coarse
and Weak Shape Categories, Mediterr. J. Math. \textbf{9} (2012), 741-766.

\noindent \lbrack 32] \textit{N. Ugle\v{s}i\'{c},} The shapes in a concrete
category, Glasnik Math. \textbf{51}(\textbf{71}) (2016), 255-306.

\noindent \lbrack 33] \textit{N. Ugle\v{s}i\'{c},} An Example Relating the
Coarse and Weak Shape, Mediterr. J. Math. \textbf{13} (2016), 4939-4947.

\noindent \lbrack 34] \textit{N. Ugle\v{s}i\'{c} and B. \v{C}ervar}, A
subshape spectrum for compacta, Glasnik Mat. \textbf{40 }(\textbf{60})
(2005) 351-388.

\noindent \lbrack 35] \textit{N. Ugle\v{s}i\'{c} and B. \v{C}ervar}, The $%
S_{n}$-equivalence of compacta, Glasnik Mat. \textbf{42}(\textbf{62})
(2007), 196-211.

\noindent \lbrack 36] N. Ugle\v{s}i\'{c} and B. \v{C}ervar, \textit{The
concept of a weak shape type}, IJPAM \textbf{39} (2007), 363-428.

\end{document}